\documentclass[a4paper,USenglish,cleveref,autoref,thm-restate,nolineno]{socg-lipics-v2021}

\hideLIPIcs

\usepackage{graphicx}
\usepackage{amsmath,amsfonts,amssymb}%
\usepackage{amsthm,mathrsfs,mathtools}%
\usepackage[T1]{fontenc}
\usepackage{csquotes}
\usepackage{tikz}
\usepackage[all,cmtip]{xy}
\usepackage{tikz-cd}
\usepackage[textsize=scriptsize]{todonotes}

\usepackage{import}

\usetikzlibrary{matrix, arrows, calc, patterns, intersections, decorations.markings, shapes.geometric, chains, decorations.pathmorphing}

\newcommand{\R}{\mathbb R}
\newcommand{\Rr}{\mathcal R}
\newcommand{\Ee}{\mathcal E}
\newcommand{\Pp}{\mathcal P}
\newcommand{\T}{\mathcal T}

\newcommand{\Uu}{\mathcal U}

\newcommand{\dFD}{d_{\mathit{FD}}}
\newcommand{\dFC}{d_{\mathit{FC}}}
\newcommand{\dU}{d_{\mathit{U}}}
\newcommand{\dI}{d_{\mathit{I}}}

\DeclareMathOperator{\pr}{pr}
\DeclareMathOperator{\im}{im}

\newlength{\xywd}
\newcommand{\xto}[2][]{%
  \sbox{0}{$\scriptstyle#1$}%
  \xywd=\wd0
  \sbox{0}{$\scriptstyle#2$}%
  \ifdim\wd0>\xywd \xywd=\wd0 \fi
  \xymatrix@C\dimexpr\xywd+1em\relax{{}\ar[r]^{#2}_{#1}&{}}%
}

\def\havard#1{\todo[color=green]{#1}}

\title{Tight Quasi-Universality of Reeb Graph Distances}

\author{Ulrich Bauer}
{Department of Mathematics and Munich Data Science Institute, Technical University of Munich (TUM),
Germany}
{ulrich.bauer@tum.de}
{https://orcid.org/0000-0002-9683-0724}
{}
\author{Håvard Bakke Bjerkevik}{Department of Mathematics, SUNY Albany, Albany, NY, USA}{hbjerkevik@albany.edu}{https://orcid.org/0000-0001-9778-0354}{Austrian Science Fund (FWF) grant number P 33765-N}
\author{Benedikt Fluhr}{Faculty of Mathematics, Bielefeld University,
Germany}{}{https://orcid.org/0000-0002-5730-0106}{}

\funding{DFG SFB/TRR 109 \emph{Discretization in Geometry and Dynamics}  -- 281071066}
 
\authorrunning{U. Bauer, H. B. Bjerkevik, and B. Fluhr} %
\Copyright{Ulrich Bauer, Håvard Bakke Bjerkevik, and Benedikt Fluhr} %

\category{} %

\EventEditors{Xavier Goaoc and Michael Kerber}
\EventNoEds{2}
\EventLongTitle{38th International Symposium on Computational Geometry (SoCG 2022)}
\EventShortTitle{SoCG 2022}
\EventAcronym{SoCG}
\EventYear{2022}
\EventDate{June 7--10, 2022}
\EventLocation{Berlin, Germany}
\EventLogo{}
\SeriesVolume{224}
\ArticleNo{13}
 
\begin{document}

\maketitle

\begin{abstract}
We establish tight bi-Lipschitz bounds certifying quasi-universality (universality up to a constant factor) for various distances between Reeb graphs: the interleaving distance, the functional distortion distance, and the functional contortion distance.  The definition of the latter distance is a novel contribution, and for the special case of contour trees we also prove strict universality of this distance.  Furthermore, we prove that for the special case of merge trees the functional contortion distance coincides with the interleaving distance, yielding universality of all four distances in this case.
\end{abstract}

\section{Introduction}

The \emph{Reeb graph} is a topological signature of real-valued functions, first considered in the context of Morse theory \cite{MR15613} and subsequently applied to the analysis of geometric shapes \cite{DBLP:conf/siggraph/HilagaSKK01,DBLP:journals/cga/ShinagawaK91}.
It describes the connected components of level sets of a function, and for Morse functions on compact manifolds or PL functions on compact polyhedra it turns out to be a finite topological graph with a function that is monotonic on the edges.
If the domain of the function is simply-connected, then the Reeb graph is contractible, hence a tree, and is therefore often called a \emph{contour tree}.
In topological data analysis, Reeb graphs are used for surveying functions, and also in a discretized form termed \emph{Mapper} \cite{DBLP:conf/spbg/SinghMC07} for the analysis of point cloud data, typically high-dimensional or given as an abstract finite metric space.

An important requirement for topological signatures is the ability to quantify their similarity, which is typically achieved by means of an extended pseudometric on the set of isomorphism classes of signatures under consideration, referred to as a distance.
In order for such a distance to be practical, it should be resilient to noise and perturbations of the input data, which is formalized by the property of \emph{stability}:
small perturbations of the data lead to small perturbations of the signature.
Mathematically speaking, the signature is a Lipschitz-continuous map between metric spaces, and often the Lipschitz constant is assumed to be $1$, meaning that the map is non-expansive.
Previous examples of distances between Reeb graphs satisfying stability include the \emph{functional distortion distance} \cite{bauer_measuring}, the \emph{interleaving distance} \cite{deSilva_categorified}, and the \emph{Reeb graph edit distance} \cite{bauer_edit}.
While stability guarantees that similarity of data sets is preserved, it does not provide any guarantees regarding the \emph{discriminativity} of the distance on the signature.
Indeed, a certain loss of information is inherent and even desired for most signatures; in fact, a key strength of topological signatures is their invariance to reparametrizations or isometries of the input data, independent of the metric used to distinguish non-isomorphic signatures.
Thus, given any signature map defined on some metric space of possible data, such as the space of real-valued functions on a fixed domain with the uniform metric, one stable distance is considered more discriminative than another if it assigns larger or equal distances to all possible pairs of signatures.
For example, the functional distortion distance is an upper bound for the interleaving distance and thus more discriminative in that sense.
The opposite relation holds up to a constant, as the two distances are bi-Lipschitz equivalent~\cite{bauer_strong_eq}.
One may now ask if a given distance is \emph{universal}, meaning that it is both stable and an upper bound for all stable distances, and thus the most discriminative among all stable distances.
This can be expressed by the universal property of a quotient metric \cite{MR2585803,scoccola2020locally}, giving rise to the name `universal'.
Since there is only one such distance, we refer to it as \emph{the} universal distance.
Perhaps surprisingly, neither the interleaving distance nor the functional distortion distance is universal, while the \emph{Reeb graph edit distance}  \cite{bauer_edit} turns out to be universal.

These results raise the question of whether the mentioned distances are \emph{quasi-universal}, i.e., bi-Lipschitz equivalent to the universal distance.
We address this question by proving lower and upper Lipschitz bounds relating all three mentioned distances, together with the novel \emph{functional contortion distance}, a slight variation of the functional distortion distance.
It has a simple definition, is more discriminative than the functional distortion and interleaving distances while still being stable, and in fact coincides with the universal distance when restricted to contour trees, as we also show in this paper.

If $d$ and $d'$ are distances on Reeb graphs and $C\in [0,\infty)$, we use the notation $d\leq Cd'$ to express that for all Reeb graphs $(F,f)$ and $(G,g)$, the inequality $d(F,G)\leq Cd'(F,G)$ holds.

\begin{theorem}[Quasi-universality of Reeb graph distances]
\label{thm:main}
The functional contortion distance $\dFC$, the functional distortion distance $\dFD$, and the interleaving distance $\dI$ on Reeb graphs are quasi-universal (bi-Lipschitz equivalent to the universal distance).
Specifically, %
we have
\begin{align*}
{\dFC} & \leq {\dU} \leq 3 {\dFC} &
{\dFD} & \leq {\dU} \leq 3 {\dFD} &
{\dI} & \leq {\dU} \leq 5 {\dI}\\
{\dFD} & \leq {\dFC} \leq 3 {\dFD} &
{\dI} & \leq {\dFC} \leq 3 {\dI} &
{\dI} & \leq {\dFD} \leq 3 {\dI}.
\end{align*}
\end{theorem}
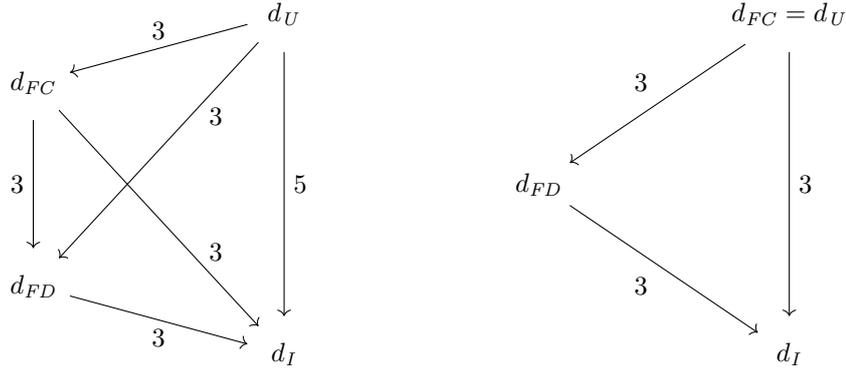
\begin{figure}[t!]
\centering
\begin{subfigure}[t]{0.34\textwidth}
\centering
\begin{tikzpicture}[scale=1.5]
\coordinate (I) at (2.2,0){};
\node at (I){$\dI$};
\coordinate (D) at (0,0.6){};
\node at (D){$\dFD$};
\coordinate (C) at (0,2.4){};
\node at (C){$\dFC$};
\coordinate (U) at (2.2,3){};
\node at (U){$\dU$};
\draw[shorten >=.5cm,shorten <=.5cm,<-] (I) to node[below] {$3$} (D);
\draw[shorten >=.5cm,shorten <=.5cm,<-] (I) to (C);
\node at (1.6,0.9){$3$};
\draw[shorten >=.5cm,shorten <=.5cm,<-] (I) to node[right] {$5$} (U);
\draw[shorten >=.5cm,shorten <=.5cm,<-] (D) to node[left] {$3$} (C);
\draw[shorten >=.5cm,shorten <=.5cm,<-] (D) to (U);
\node at (1.6,2.1){$3$};
\draw[shorten >=.5cm,shorten <=.5cm,<-] (C) to node[above] {$3$} (U);
\end{tikzpicture}
\caption{Inequalities for Reeb graphs.}
\label{Fig:ineq}
\end{subfigure}\hspace{2cm}
\begin{subfigure}[t]{0.34\textwidth}
\centering
\begin{tikzpicture}[scale=1.5]
\coordinate (I) at (2.2,0){};
\node at (I){$\dI$};
\coordinate (D) at (0,1.5){};
\node at (D){$\dFD$};
\coordinate (U) at (2.2,3){};
\node at (U){$\dFC = \dU$};
\draw[shorten >=.5cm,shorten <=.5cm,<-] (I) to (D);
\node at (0.9,0.6){$3$};
\draw[shorten >=.5cm,shorten <=.5cm,<-] (I) to node[right] {$3$} (U);
\draw[shorten >=.7cm,shorten <=.5cm,<-] (D) to (U);
\node at (0.9,2.4){$3$};
\end{tikzpicture}
\caption{Inequalities for contour trees.}
\label{Fig:ineq_contour}
\end{subfigure}\hspace{2cm}
\caption{An arrow from $d$ to $d'$ labeled $C$ means that the inequalities $d'(F,G) \leq d(F,G)\leq Cd'(F,G)$ hold for all Reeb graphs (in (a)) or contour trees (in (b)).}
\end{figure}
Of these inequalities, only $\dI \leq \dFD \leq 3 \dI$ \cite[Theorem 16]{bauer_strong_eq}, $\dI\leq \dU$, and $\dFD\leq \dU$ were known, the latter two being equivalent to stability of $\dI$ \cite[Theorem 4.4]{deSilva_categorified} and of $\dFD$ \cite[Theorem 4.1]{bauer_measuring}.
See \cref{Fig:ineq} for a visualization of the inequalities in \cref{thm:main}.
One can easily check that all the inequalities in the theorem follow from only six of them, namely 
\begin{align*}
\dI &\leq \dFD \leq \dFC \leq \dU &
\dU &\leq 3\dFD&
\dU &\leq 5\dI&
\dFC &\leq 3\dI.
\end{align*}
We further show that all these bounds are indeed tight. Here tightness of a bound $d\leq Cd'$ means that $C = \inf \{ C' \mid d \leq C' d' \}$;
equivalently, $C = \sup \{ \frac{d(x,y)}{d'(x,y)} \mid d'(x,y)>0 \}$.
Previously, only the stability bounds $\dI\leq \dFD\leq \dU$ were known to be tight. Furthermore, the interleaving and functional distortion distances were known not to be universal. Specifically, an example in \cite[Propositions~7~and~8]{bauer_edit} shows that ${\dU}$ and ${\dFD}$ can differ by a factor of $2$, and hence
$\inf \{ C \mid {\dU} \leq C {\dFD} \} \geq 2$.

We establish sharper bounds for the special case of contour trees, leading to universality of the functional contortion distance.

\begin{restatable}[Universality of the functional contortion distance for contour trees]{theorem}{thmContour}
\label{thm:contour}
The functional contortion distance is universal for contour trees:
given two contour trees $(F,f)$ and $(G,g)$, we have \[{\dFC(F,G)} = {\dU(F,G)}.\]
\end{restatable}
This theorem gives us a simpler set of inequalities for contour trees than what we have for general Reeb graphs; see \cref{Fig:ineq_contour}.

Finally, we show that the interleaving distance of merge trees, considered as a special case of contour trees with only downward branches, coincides with the functional contortion distance, establishing the universality of all four distances in this particular setting.
\begin{restatable}[Universality of the interleaving distance for merge trees]{theorem}{thmMerge}
\label{thm:merge}
The interleaving distance is universal for merge trees:
given two merge trees $(F,f)$ and $(G,g)$, we have \[{\dI(F,G)} = {\dFD(F,G)} = {\dFC(F,G)} = {\dU(F,G)}.\]
\end{restatable}

Previously, only the equality ${\dI(F,G)} = {\dFD(F,G)}$ was known \cite[Theorem 6.2]{bauer_measuring_arxiv}.
We prove \cref{thm:main} in \cref{sec:reeb}, \cref{thm:contour} in \cref{sec:contour}, and \cref{thm:merge} in \cref{sec:merge}.

As a corollary of \cref{thm:main}, we get that any distance that is bi-Lipschitz equivalent to the interleaving distance is quasi-universal.
In particular, this property holds for the truncated interleaving distances introduced by Chambers et al. \cite[Cor. 5.4]{chambers_truncated}.
\begin{corollary}
For $m\in [0,1)$, the distance $\dI^m$ in \cite{chambers_truncated} is quasi-universal on Reeb graphs with finite support.
\end{corollary}

In \cite{cardona_et_al:LIPIcs.SoCG.2022.24}, an $\ell^p$-generalization of the interleaving distance is introduced for unbounded merge trees with finitely many nodes, for all $p\in [1,\infty]$.  It satisfies a universal property analogous to the one considered here.  The case $p=\infty$ yields a variant of our universality result for the interleaving distance between unbounded merge trees with finitely many nodes.

\section{Preliminaries}
\label{sec:preliminaries}

\begin{definition}[Reeb graph]
  \label{def:ReebGraph}
A Reeb graph is a pair $(F,f)$ where $F$ is a non-empty connected topological space and $f \colon F \to \R$ a continuous function, such that $F$ admits the structure of a $1$-dimensional CW complex for which
\begin{itemize}
\item $f$ restricts to an embedding on each $1$-cell, and
\item for every bounded interval $I\subset \R$, the preimage $f^{-1}(I)$ intersects a finite number of cells.
\end{itemize}
\end{definition}
We often refer to $F$ as a Reeb graph without referring explicitly to the function $f$.
A morphism (isomorphism) of Reeb graphs is a value-preserving continuous map (homeomorphism).

\begin{remark}
  \label{rem:pathConn}
  Suppose $(F, f)$ is a Reeb graph,
  let $I \subseteq \R$ be a closed interval,
  and fix the structure of a CW complex on $F$ as in \cref{def:ReebGraph}.
  As we may subdivide any $1$-cell
  whose interior intersects $f^{-1}(\partial I)$,
  the preimage $f^{-1}(I)$ also admits the structure of a CW complex.
  Thus, the preimage $f^{-1}(I)$ is locally path-connected
  and therefore the connected components and the path-components
  of $f^{-1}(I)$ coincide.
\end{remark}

\begin{definition}[Contour tree]
A contour tree is a contractible Reeb graph.
\end{definition}

A path in a space $X$ is a function $\rho\colon [0,1] \to X$.
We say that $\rho$ goes from $\rho(0)$ to $\rho(1)$, and we sometimes abuse terminology by identifying $\rho$ and $\im \rho$, so that we consider a path as a subspace of $X$ instead of a function.

Contour trees admit a minimal connected subspace containing any set of points, which is a property we will make good use of. 
\begin{definition}
\label{def:B(X)}
Let $F$ be a contour tree, and let $x, x'\in F$.
Since $F$ is Hausdorff, there is an injective path $\rho$ from $x$ to $x'$.
We define the \emph{connecting segment of $x$ and $x'$ in $F$} as
\begin{equation*}%
B(x,x')= \im\rho.
\end{equation*}
If $x=x'$, we let $B(x,x') = \{x\}$.
For any subset $X\subseteq F$, we define the \emph{connecting subtree of $X$ in $F$} as
\begin{equation*}%
B(X)= \bigcup_{x,x'\in X} B(x,x').
\end{equation*}
\end{definition}
Since $F$ is a contractible $1$-dimensional CW-complex, $B(x,x')$ contains any path between $x$ and $x'$.
One can show that $B(X)$ is (path-)connected (if $y\in B(x_1,x_2)$ and $z\in B(x'_1,x'_2)$, use $B(x_1,x_2)$, $B(x_2,x'_1)$ and $B(x'_1,x'_2)$), and that any connected subspace $C\subseteq H$ containing $X$ must contain $B(X)$ ($C$ must contain a path from $x$ to $x'$ for any $x,x'\in X$).

Contour trees further specialize to \emph{merge trees},
which can be thought of as upside down trees
in the sense that
its branches grow from top to bottom.

\begin{definition}[Merge tree]
  \label{def:mergeTree}
  A \emph{merge tree} is a Reeb graph $(F,f)$,
  such that $F$ admits the structure of a $1$-dimensional CW complex
  as in \cref{def:ReebGraph}
  with the additional property
  that each $0$-cell is the lower boundary point of at most a single $1$-cell.
\end{definition}

We note that this definition allows for both
unbounded merge trees,
which is an implied necessity of the definition in \cite{MBW13},
and bounded merge trees.
We verify that merge trees are indeed special cases of contour trees:

\begin{lemma}
Any merge tree is a contour tree.
\end{lemma}

\begin{proof}
Assuming $(F,f)$ is a merge tree, we need to show that $F$ is contractible.
By the Whitehead theorem, it suffices to show that all homotopy groups except in degree zero are trivial.
For degree at least two, this follows from $F$ being $1$-dimensional, which leaves degree $1$.
Let $\rho: S^1\to F$ be any continuous map.
The image of $\rho$ lies in $f^{-1}(I)$ for some finite interval $I$, so by definition of Reeb graph, it only intersects a finite number of cells.
Let $T$ be the minimal subcomplex of $F$ containing these cells.
It follows from the property of each $0$-cell being the lower boundary of at most one $1$-cell that $T$ is a (finite) tree.
This means that $\rho$ is homotopic to a constant map.
Since $\rho$ was arbitrary, we get that the fundamental group of $F$ is trivial, which completes the proof.
\end{proof}

\begin{definition}[Induced Reeb graph]
\label{def:indReebGraph}
Let $X$ be a topological space, $f\colon X\to \R$ a continuous function.
Let $\Rr X$ be the quotient space
$X / {\sim_f}$,
with $x \sim_f y$ iff $x$ and $y$
belong to the same connected component of some level set
of $f \colon X \rightarrow \R$, and
let $q_{\Rr X} \colon X \rightarrow \Rr X$
be the natural quotient map,
and let $\Rr f \colon \Rr X \rightarrow \R$
be the unique continuous map such that the diagram
\begin{equation*}
\begin{tikzcd}
X
\arrow[rr, "q_{\Rr X}", bend left]
\arrow[rd, "f"']
&
&
\Rr X
\arrow[ld, "\Rr f"]
\\
&
\R
\end{tikzcd}
\end{equation*}
commutes. If $(\Rr X, \Rr f)$ is a Reeb graph, we say that it is the Reeb graph \emph{induced by $(X,f)$}.
\end{definition}

We say that two or more points are \emph{connected in} some space if there is a connected component of that space containing all those points.
For $t \in \R$ and $\delta \ge 0$, we write 
\[
[t\pm\delta] := [t-\delta,t+\delta] \subset \R
\]
for the closed interval of radius $\delta$ centered at $t$.

\sloppy
Let $f \colon X \rightarrow \R$ be a continuous function
and let $\delta \geq 0$.
  We define the \emph{$\delta$-thickening} of $X$ as 
  \begin{align*}
  \mathcal{T}_{\delta} X &:= X \times [-\delta, \delta]
  ,&
  \mathcal{T}_{\delta} f &\colon \mathcal{T}_{\delta} X \rightarrow \R, \,
  (p, t) \mapsto f(p) + t.
  \end{align*}
  Moreover,
  let
  \[{\tau^{\delta}_X \colon X \rightarrow \mathcal{T}_{\delta} X,\,
  p \mapsto (p, 0)}\]
  be the natural embedding of $X$ into its $\delta$-thickening.
Now let $g \colon G \rightarrow \R$ be a Reeb graph
and consider its $\delta$-thickening
$\mathcal{T}_{\delta} g \colon \mathcal{T}_{\delta} G \rightarrow \R$.
  We define the \emph{$\delta$-smoothing} of $G$ as 
  \begin{align*}
  \mathcal{U}_{\delta} G &:= \mathcal{R} \mathcal{T}_{\delta} G
  ,&
  \mathcal{U}_{\delta} g &:= \mathcal{R} \mathcal{T}_{\delta} g
  \colon \mathcal{R} \mathcal{T}_{\delta} G \rightarrow \R.
  \end{align*}
  Moreover,
  let
  \[q_{\mathcal{U}_{\delta} G} \colon
    \mathcal{T}_{\delta} G \rightarrow \mathcal{U}_{\delta} G =
    \mathcal{R} \mathcal{T}_{\delta} G
  \]
  be the natural quotient map as in \cref{def:indReebGraph},
  and let
  \[\zeta^{\delta}_G \colon
    G \xto{\tau^{\delta}_G}
    \mathcal{T}_{\delta} G \xto{q_{\mathcal{U}_{\delta} G}}
    \mathcal{U}_{\delta} G
  \]
  be the natural map of $G$ into its $\delta$-smoothing,
  which is the composition of $\tau^{\delta}_G$
  and $q_{\mathcal{U}_{\delta} G}$.

Now suppose $f \colon F \rightarrow \R$
is another Reeb graph
and that $\phi \colon F \rightarrow \mathcal{U}_{\delta} G$
is a continuous map.
Identifying points in $\mathcal{U}_{\delta} G = \mathcal{R} \mathcal{T}_{\delta} G$ with subsets of $\mathcal{T}_{\delta} G$ via the quotient map $q_{\mathcal{U}_{\delta} G} \colon \mathcal{T}_{\delta} G \to \mathcal{U}_{\delta} G$,
the map $\phi$ induces a set-valued map
\[
  \Phi := \pr_G{} %
  \circ \phi \colon F \to \mathcal P(G) 
\]
from $F$ to the power set of $G$,
where
$\operatorname{pr}_G \colon
\mathcal{T}_{\delta} G = G \times [-\delta, \delta] \rightarrow G,\,
(q, t) \mapsto q$
is the projection onto $G$.
Moreover, suppose
$\psi \colon G \rightarrow \mathcal{U}_{\delta} F$
is another continuous map, and define the set-valued map analogously as
\begin{equation*}
  \Psi := \pr_F{} %
  \circ \psi \colon G \to \mathcal P(F).
\end{equation*}

\begin{definition}[Interleaving distance $\dI$ \cite{deSilva_categorified}]
  \label{def:interleaving}
  We say that the pair of maps
  $\phi \colon F \rightarrow \mathcal{U}_{\delta} G$
  and
  $\psi \colon G \rightarrow \mathcal{U}_{\delta} F$
  is a \mbox{\emph{$\delta$-interleaving}} of $(F, f)$ and $(G, g)$
  if the triangles
  
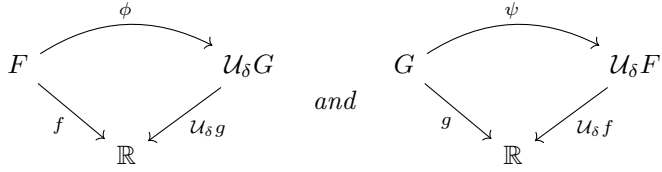
\begin{equation*}
    \begin{tikzcd}
      F
      \arrow[rr, bend left, "\phi"]
      \arrow[rd, "f"']
      & &
      \mathcal{U}_{\delta} G
      \arrow[ld, "\mathcal{U}_{\delta} g"]
      \\
      &
      \R
    \end{tikzcd}
    \quad
    \text{and}
    \quad
    \begin{tikzcd}
      G
      \arrow[rr, bend left, "\psi"]
      \arrow[rd, "g"']
      & &
      \mathcal{U}_{\delta} F
      \arrow[ld, "\mathcal{U}_{\delta} f"]
      \\
      &
      \R
    \end{tikzcd}
  \end{equation*}
  commute and the following two conditions are satisfied:
  \begin{itemize}
  \item
    For any $x \in F$, $x$ and $\Psi(\Phi(x))$
    are connected in
    $f^{-1} [f(x)\pm 2\delta]$.
  \item
    For any $y \in G$, $y$ and $\Phi(\Psi(y))$ are connected in
    $g^{-1} [g(y)\pm 2\delta]$.
  \end{itemize}
  The \emph{interleaving distance},
  denoted $\dI(F,G)$,
  is defined as the infimum of the set of $\delta$
  admitting a $\delta$-interleaving between $(F,f)$ and $(G,g)$.
\end{definition}

Note that for any $t \in \R$ the map $f^{-1}[t\pm\delta] \to (\mathcal{T}_{\delta} f)^{-1}(t)$ given by $x \mapsto (x, t - f(x))$ is a homeomorphism, with the inverse given by the restriction of $\pr_F \colon \mathcal{T}_{\delta} F \to F$.
In particular, the points of $\mathcal{U}_{\delta} F$, which are the connected components of level sets of $\mathcal{T}_{\delta} f$, are in bijection with connected components of interlevel sets of $f$.
Hence, the connectedness condition for an interleaving is equivalent to requiring that 
$\psi_\delta \circ \phi = \tau^{2\delta}_F$ and $\phi_\delta \circ \psi = \tau^{2\delta}_G$, where $\phi_\delta$ is the induced map $\mathcal{U}_{\delta} F \to \mathcal{U}_{2\delta} G, \, [x,s] \mapsto [\phi(x),s]$ and similarly for $\psi_\delta$.

\begin{definition}[Functional distortion distance $\dFD$ \cite{bauer_measuring}]
\label{Def:dFD}
Let $(F,f)$ and $(G,g)$ be two Reeb graphs.
Given a pair $(\phi,\psi)$ of maps $\phi \colon F \to G$ and $\psi \colon G \to F$,
consider the correspondence
\[
  C(\phi,\psi) = \{ (x,y) \in F \times G \mid \phi(x) = y \textrm { or } x = \psi(y)\}.
\]
The pair $(\phi,\psi)$ is a \emph{$\delta$-distortion pair} if $\|f-g\circ\phi\|_\infty\leq \delta$, $\|f\circ\psi-g\|_\infty\leq \delta$, and for all
$
  (x,y),(x',y')\in C(\phi,\psi)
$
we have
\[
  \left| d_f(x,x') - d_g(y,y')\right| \leq 2\delta,
\]
where $d_f(x,x')$ denotes the infimum length of any interval $I$ such that $x$ and $x'$ are connected in $f^{-1}(I)$, and similarly for $d_g$.
The \emph{functional distortion distance}, denoted $\dFD(F,G)$, is defined as the infimum of all $\delta$ admitting a $\delta$-distortion pair between $(F,f)$ and $(G,g)$.
\end{definition}

\begin{definition}[Functional contortion distance $\dFC$]
\label{def:contortionDistance}
Let $(F,f)$ and $(G,g)$ be two Reeb graphs.
A pair $(\phi,\psi)$ of maps $\phi \colon F \to G$ and $\psi \colon G \to F$ is a \emph{$\delta$-contortion pair} between $(F,f)$ and $(G,g)$ if the following symmetric conditions are satisfied.
\begin{itemize}
\item For any $x \in F$ and $y\in \psi^{-1}(x)$, the points $\phi(x)$ and $y$ are connected in $g^{-1}[f(x)\pm\delta]$.
\item For any $y \in G$ and $x\in \phi^{-1}(y)$, the points $\psi(y)$ and $x$ are connected in $f^{-1}[g(y)\pm\delta]$.
\end{itemize}
The \emph{functional contortion distance}, denoted $\dFC(F,G)$, is defined as the infimum of the set of $\delta$ admitting a $\delta$-contortion pair between $(F,f)$ and $(G,g)$.
\end{definition}
The definition of $\dFC$ is arguably easier to work with than that of $\dFD$, since to verify that $(\phi,\psi)$ is a $\delta$-contortion pair, one only has to check one condition for each element of $C(\phi,\psi)$, while to verify that $(\phi,\psi)$ is a $\delta$-distortion pair, one needs to check a condition for each pair of elements of $C(\phi,\psi)$.

\begin{proposition}
The functional contortion distance $\dFC$ satisfies the triangle inequality.
\end{proposition}

\begin{proof}
Suppose that $(F,f)$, $(G,g)$, and $(H,h)$ are Reeb graphs, $(\phi,\psi)$ is a $\epsilon$-contortion pair between $(F,f)$ and $(G,g)$, and $(\mu,\nu)$ is a $\delta$-contortion pair between $(G,g)$ and $(H,h)$.
To prove that $\dFC$ satisfies the triangle inequality, it suffices to show that $(\mu\circ \phi,\psi\circ \nu)$ is a $(\epsilon+\delta)$-contortion pair between $(F,f)$ and $(G,g)$.

Let $x\in F$ and $y\in \psi\circ \nu(x)$.
We need to show that $y$ and $\mu\circ \phi(x)$ are connected in $h^{-1}[f(x)\pm(\epsilon+\delta)]$; the second condition in the definition of $\dFC$ follows by symmetry.
Since $(\phi,\psi)$ is a $\epsilon$-contortion pair, we know that $\nu(y)$ and $\phi(x)$ are connected in $g^{-1}[f(x)\pm\epsilon]$.
Let $K$ be the connected component of $g^{-1}[f(x)\pm\epsilon]$ containing $\nu(y)$ and $\phi(x)$.
Then $\mu(K)$ contains $\mu\circ\nu(y)$ and $\mu\circ \phi(x)$, it is connected, and by \cref{Rem:cont_infty_bound}, $\mu(K)\subseteq h^{-1}[f(x)\pm(\epsilon+\delta)]$.
Because $(\mu,\nu)$ is a $\delta$-contortion pair, $y$ and $\mu\circ\nu(y)$ are connected in $h^{-1}[g(\nu(y))\pm\delta]\subseteq h^{-1}[f(x)\pm(\epsilon+\delta)]$.
It follows that $y$ and $\mu\circ\phi(x)$ are connected in $h^{-1}[f(x)\pm(\epsilon+\delta)]$, which concludes the proof.
\end{proof}

\begin{remark}
\label{Rem:cont_infty_bound}
Let $(\phi,\psi)$ be a $\delta$-contortion pair between $(F,f)$ and $(G,g)$.
For any $x\in F$ we have $\phi(x)\in g^{-1}[f(x)\pm\delta]$, which implies $\|f(x)-g\circ \phi(x)\|\leq \delta$.
It follows that $\|f-g\circ\phi\|_\infty\leq \delta$, and by a symmetric argument we also get $\|g-f\circ\psi\|_\infty\leq \delta$.
\end{remark}

\begin{figure}
\begin{tikzpicture}[scale=0.8]
\draw (-2,-3) to (-2,3);
\draw (-2.1,-2) to (-1.9,-2);
\node[left] at (-2.1,-2){$-2$};
\draw (-2.1,-1) to (-1.9,-1);
\node[left] at (-2.1,-1){$-1$};
\draw (-2.1,-0) to (-1.9,0);
\node[left] at (-2.1,0){$0$};
\draw (-2.1,1) to (-1.9,1);
\node[left] at (-2.1,1){$1$};
\draw (-2.1,2) to (-1.9,2);
\node[left] at (-2.1,2){$2$};
\coordinate (x) at (-.65,0){};
\node[left] at (x){$x$};
\coordinate (z) at (.65,0){};
\node[left] at (z){$z$};
\coordinate (y) at (4,0){};
\node[right] at (y){$y$};
\foreach \X in {x,z,y}
\draw[color=black,fill=black] (\X) circle (.05);
\draw[thick] (0,-3) to (0,-2);
\draw[thick] (0,2) to (0,3);
\draw[thick,color=red] (x) to [out=-90,in=120] (0,-2) to [out=60,in=-90] (z);
\draw[thick] (x) to [out=90,in=-120] (0,2) to [out=-60,in=90] (z);
\node[above] at (0,3){$F$};
\draw[thick] (4,-3) to (4,3);
\node[above] at (4,3){$G$};
\draw[->,shorten >=.15cm,shorten <=.15cm] (x) to [out=25,in=155] (y);
\draw[->,shorten >=.15cm,shorten <=.15cm] (y) to [out=205,in=-25] (z);
\end{tikzpicture}
\caption{Reeb graphs $(F,f)$ and $(G,g)$. If $\phi\colon F\to G$ and $\psi\colon G\to F$ are a $1$-distortion pair, we allow $\phi(x)=y$ and $\psi(y)=z$ because there is a path from $x$ to $z$ in $f^{-1}[-2,0]$. However, in this case $(\phi,\psi)$ is not a $1$-contortion pair, because $x$ and $z$ are not connected in $f^{-1}[g(y)\pm1] = f^{-1}[-1,1]$.}
\end{figure}
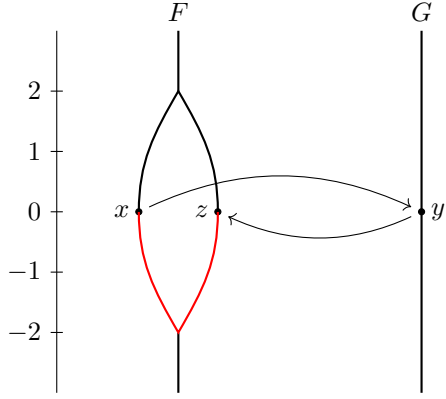

\begin{definition}[Universal distance $\dU$ \cite{MR2585803,bauer_edit}]
\label{def:universalDistance}
Let $(F,f)$ and $(G,g)$ be two Reeb graphs.
The \emph{universal distance}, denoted $\dU(F, G)$, is defined as the infimum
of $\|\tilde f-\tilde g\|_\infty$
taken over all spaces~$Z$ and functions $\tilde f, \tilde g \colon Z \to \R$ such that $(\Rr Z, \Rr \tilde f) \cong (F,f)$ and $(\Rr Z, \Rr \tilde g) \cong (F,g)$. 
\end{definition}

The distance $\dU$ is readily seen to be universal.
Recall that the \emph{Reeb graph edit distance}~\cite{bauer_edit} is also universal, providing an alternative explicit construction for the universal distance.

\section{Bi-Lipschitz bounds for Reeb graph distances}
\label{sec:reeb}

This section is devoted to proving \cref{thm:main}.
In \cref{subsec:easy_ineq} we prove $\dI\leq \dFD\leq \dFC\leq \dU$, in \cref{subsec:U<=3FD} we prove $\dU\leq 3\dFD$, in \cref{subsec:U<=5I} we prove $\dU\leq 5\dI$, and in \cref{subsec:FC<=3I} we prove $\dFC\leq 3\dI$.
As mentioned before, these inequalities together imply the theorem.

\subsection{The inequalities $\dI\leq \dFD\leq \dFC\leq \dU$}
\label{subsec:easy_ineq}

The following lemma and its proof are similar to \cite[Proposition 3]{bauer_edit}, but without the assumption that $X$ is compact.
\begin{lemma}
\label{Lem:connected_preimage}
Let $(F,f)$ be a Reeb graph induced by a map $\hat f\colon X\to \R$, let $q\colon X\to F$ be the associated quotient map, and suppose $K\subseteq F$ is connected. Then $q^{-1}(K)$ is connected.
\end{lemma}

\begin{proof}
Suppose the lemma is false. Then $q^{-1}(K)= X_1 \sqcup X_2$, where $X_1$ and $X_2$ are nonempty and contained in disjoint open subsets of $X$.
Since $F$ is equipped with the quotient topology of $q$, $q(X_1)$ and $q(X_2)$ are both open subsets of $q(X_1)\cup q(X_2)=K$.
Because $K$ is connected, we must have $q(X_1)\cap q(X_2)\neq \emptyset$.

Let $x\in q(X_1)\cap q(X_2)$, so $V_1\coloneqq q^{-1}(x)\cap X_1$ and $V_2\coloneqq q^{-1}(x)\cap X_2$ are both nonempty.
Since $X_1$ and $X_2$ are open and disjoint subsets of $X_1\cup X_2$, $V_1$ and $V_2$ are open and disjoint subsets of $V_1\cup V_2=q^{-1}(x)$.
But by definition of induced Reeb graph, $q^{-1}(x)$ is connected, so we have a contradiction.
\end{proof}

\begin{theorem}
\label{Thm:FC<=U}
Given any two Reeb graphs $(F,f)$ and $(G,g)$, we have
\[\dFC(F,G) \leq \dU(F,G).\]
\end{theorem}
\begin{proof}
Let $X$ be a topological space with functions $\hat f,\hat g\colon X\to \R$ that induce Reeb graphs $(F,f)$ and $(G,g)$, respectively, and let $q_F \colon X \to F$, $q_G \colon X \to G$ denote the corresponding Reeb quotient maps. Suppose $\|\hat f-\hat g\|_\infty = \delta \geq 0$.
For any $\epsilon>0$, we will construct functions $\phi\colon F\to G$ and $\psi\colon G\to F$ that form a $(\delta+2\epsilon)$-contortion pair; the theorem follows.

Fix $\epsilon>0$. Pick a discrete subset $S\subseteq F$ containing all the $0$-cells of $F$ such that for each $1$-cell~$C$ and each connected component $K$ of $C\setminus S$, the image $f(K)$ is contained in some interval $[a,b]$ of length $b-a = \epsilon$.
Pick a subset $T\subseteq G$ analogously.
Define a map $\phi\colon S\to G$ by picking an element $\phi(s)\in q_G(q_F^{-1}(s))$ for each $s\in S$.

Let $L$ be the closure of a connected component of $F\setminus S$. Observe that $L$ is contained in a single $1$-cell $C$ and is homeomorphic to a closed interval, with endpoints $z,z'\in S$.
By our assumptions on $S$, $L$ is contained in a connected component of $f^{-1}[a,b]$ for some $a<b$ with $b-a = \epsilon$.
By \cref{Lem:connected_preimage}, $q_F^{-1}(L)$ is connected, and by continuity of $q_G$, $J\coloneqq q_G(q_F^{-1}(L))$ is connected, too.
Since $J$ is connected, we can extend $\phi$ continuously to $L$ by choosing a path from $\phi(z)$ to $\phi(z')$ in $J$.
Moreover, because $\|\hat f-\hat g\|_\infty = \delta$, we have $J\subseteq g^{-1}[a-\delta,b+\delta] \subseteq g^{-1}[f(x)\pm(\delta+\epsilon)]$.
It follows that for every $x\in L$, $\phi(x)$ and $q_G(q_F^{-1}(x))$ are connected in $g^{-1}[f(x)\pm(\delta+\epsilon)]$.
We do this for every $L$ as described and get a continuous map $\phi\colon F\to G$.
Analogously, we get a continuous map $\psi\colon G\to F$ such that for any $y\in G$, $\psi(y)$ and $q_F(q_G^{-1}(y))$ are connected in $f^{-1}[g(y)\pm(\delta+\epsilon)]$.

Pick an $x\in L$, where $L$ is as in the previous paragraph, and let $y=\phi(x)$.
By construction, $y\in q_G(q_F^{-1}(x'))$ for some $x'\in L$.
Thus, $x'\in q_F(q_G^{-1}(y))$, which, as noted, is in the same connected component of $f^{-1}[g(y)\pm(\delta+\epsilon)]$ as $\psi(y)$.
But $x$ and $x'$ are connected in $f^{-1}[a,b]$ for some $a<b$ with $b-a = \epsilon$, so it follows that $x$ and $\psi(y)$ are connected in $f^{-1}[g(y)\pm(\delta+2\epsilon)]$.
Along with the symmetric statement that follows by a similar argument, this is exactly what is needed for $(\phi,\psi)$ to be a $(\delta+2\epsilon)$-contortion pair.
\end{proof}

\begin{theorem}
\label{Thm:FD<=FC}
Given any two Reeb graphs $(F,f)$ and $(G,g)$, we have
\[\dFD(F,G) \leq \dFC(F,G).\]
\end{theorem}
This theorem follows immediately from the following lemma:
\begin{lemma}
\label{lem:cont_is_dist}
Suppose $\phi\colon F\to G$ and $\psi\colon G\to F$ form a $\delta$-contortion pair.
Then $\phi$ and $\psi$ also form a $\delta$-distortion pair.
\end{lemma}
\begin{proof}
Suppose $\phi\colon F\to G$ and $\psi\colon G\to F$ form a $\delta$-contortion pair for some $\delta\geq 0$. Then
\begin{equation}
\label{Eq:FD<=FC}
||f-g\circ\phi||_\infty, ||g-f\circ\psi||_\infty\leq \delta
\end{equation}
by \cref{Rem:cont_infty_bound}. Let $(x,y),(x',y')\in C(\phi,\psi)$, where $C(\phi,\psi)$ is as in \cref{Def:dFD}.
We claim that if $x$ and $x'$ are connected in $f^{-1}[a,b]$ for some $a\leq b$, then $y$ and $y'$ are connected in $g^{-1}[a-\delta,b+\delta]$.
Together with the symmetric statement and \cref{Eq:FD<=FC}, this is enough to show that $(\phi,\psi)$ is a $\delta$-distortion pair, from which the lemma follows.

Assume that $x$ and $x'$ lie in the same connected component $K$ of $f^{-1}[a,b]$.
We have that $\phi(K)$ is connected, and it follows from \cref{Eq:FD<=FC} that $\phi(K)\subseteq g^{-1}[a-\delta,b+\delta]$.
Since $\phi(x),\phi(x')\in \phi(K)$, $\phi(x)$ and $\phi(x')$ are connected in $g^{-1}[a-\delta,b+\delta]$.
By definition of $C(\phi,\psi)$, either $y$ is equal to $\phi(x)$, or $y\in \psi^{-1}(x)$.
In the latter case, $y$ and $\phi(x)$ are connected in $g^{-1}[f(x)\pm\delta]\subseteq g^{-1}[a-\delta,b+\delta]$ by definition of $\delta$-contortion pair.
Similarly, $y'$ and $\phi(x')$ are also connected in $g^{-1}[a-\delta,b+\delta]$.
Putting everything together, $y$ and $y'$ are connected in $g^{-1}[a-\delta,b+\delta]$, which completes the proof.
\end{proof}

\begin{theorem}[{\cite[Lemma 8]{bauer_strong_eq}}]
\label{Thm:I<=FD}
Given any two Reeb graphs $(F,f)$ and $(G,g)$, we have
\[\dI(F,G) \leq \dFD(F,G).\]
\end{theorem}
The setting of \cite{bauer_strong_eq} is slightly more restrictive than ours (considering only finite Reeb graphs), but the proof of the result applies verbatim.

\subsection{Relating universal and functional distortion distance}
\label{subsec:U<=3FD}

We denote the connected component of a point $p$ in a space $X$ by $K_p(X)$.

\begin{theorem}
\label{Thm:U<=3FD}
Given any two Reeb graphs $(F,f)$ and $(G,g)$, we have
\[\dU(F,G) \leq 3 \dFD(F,G).\]
\end{theorem}

\begin{proof}
Assume that $\phi:F \to G$ and $\psi:G \to F$ form a $\delta$-distortion pair.
We construct a subspace $Z \subseteq F \times G$ such that the canonical projections $\pr_F \colon F \times G \to F$, $\pr_G \colon F \times G \to G$ restrict to Reeb quotient maps $q_F \colon Z \to F$, $q_G \colon Z \to G$ of $f \circ q_F$ and $g \circ q_G$, and $\|f \circ q_F - g \circ q_G\|_\infty \leq 3\delta$, proving that $\dU \leq 3\dFD$.

For $x\in F$, let
\[
C(x) = K_x(f^{-1}[a,a+2\delta]),
\]
where $a$ is chosen such that $C(x)$ contains $\psi\circ\phi(x)$. By definition of $\delta$-distortion pair, such an $a$ always exists, though it does not have to be unique. We define $C(y)$ analogously for $y\in G$:
\[
C(y) = K_y(g^{-1}[a',a'+2\delta])
\]
for some $a'$, and $C(y)$ contains $\phi\circ\psi(y)$. Now define
\[
Z = \bigcup_{x \in F} C(x) \times \phi(C(x)) \cup \bigcup_{y \in G} \psi(C(y)) \times C(y)
\subseteq F \times G
\]
and the functions $\hat f = f \circ \pr_F$, $\hat g = g \circ \pr_G \colon Z \to \R$.

To show that $\|\hat f - \hat g\|_\infty \leq 3\delta$, by symmetry it suffices to show that for every $x \in F$ and every $(z,y) \in C(x) \times \phi(C(x))$ we have $|f(z)-g(y)| \leq 3\delta$.
Pick $w\in C(x)$ such that $\phi(w)=y$. We have $|f(z)-f(w)|\leq 2\delta$ by construction of $C(x)$, and $|f(w)-g(y)|\leq \delta$ by definition of $\delta$-distortion pair.
Together, we have $|f(z)-g(y)| \leq 3\delta$ as claimed.

To show that $q_F\colon Z\to F$ is surjective, simply observe that for any $x\in F$,
\[
(x,\phi(x))\in C(x)\times \phi(C(x))\subseteq Z.
\]
A similar argument shows that also $q_G\colon Z\to G$ is surjective.

It remains to show that the fibers of $q_F$ are connected; by symmetry, the same is then true for $q_G$ as well.
The fiber of $z \in F$ is of the form $q_F^{-1}(z) = \{z\} \times G_z \subseteq Z$, where $G_z = q_G(q_F^{-1}(z)) \subseteq G$ is a subspace, homeomorphic to the fiber.
Note that $G_z$ has the explicit description
\[
G_z = \bigcup_{\substack{x \in F\\z \in C(x)}} \phi(C(x)) \cup \bigcup_{\substack{y \in G\\z \in \psi(C(y))}} C(y).
\]
Now $\phi(z)$ is contained in any $\phi(C(x))$ with $x \in F$ and $z \in C(x)$,
and in any $C(y)$ with $y \in \psi^{-1}(z)$, and each of these subspaces is connected. Thus,
\[
G'_z = \bigcup_{\substack{x \in F\\z \in C(x)}} \phi(C(x)) \cup \bigcup_{\substack{y \in \psi^{-1}(z)\\z \in \psi(C(y))}} C(y)
\]
is connected and contains $\psi^{-1}(z)$ as a subset.
Clearly, if $z \in \psi(C(y))$, then $C(y)$ contains an element of $\psi^{-1}(z)$, so $C(y)$ intersects $G'_z$. As $C(y)$ is connected, it follows that
\[
G'_z \cup \bigcup_{\substack{y \in G\\z \in \psi(C(y))}} C(y) = G_z
\]
is connected.
\end{proof}

\subsection{Relating universal and interleaving distance}
\label{subsec:U<=5I}

\begin{lemma}
\label{Lem:Phi_preserves_conn}
Let $(\phi,\psi)$ be a $\delta$-interleaving of $(F,f)$ and $(G,g)$ for some $\delta\geq 0$. If $K\subseteq F$ ($K'\subseteq G$) is connected, then $\Phi(K)$ ($\Psi(K')$) is connected.
\end{lemma}

\begin{proof}
By continuity of $\phi$, $\phi(K)\subseteq \mathcal{U}_\delta G$ is connected.
Thus, by \cref{Lem:connected_preimage},
\[
C\coloneqq q_{\mathcal{U}_\delta G}^{-1}(\phi(K))\subseteq \mathcal{T}_\delta G
\]
is connected.
We have that $\Phi(K)$ is exactly the image of $C$ under the projection $\pr_G\colon \mathcal{T}_\delta G\to G$.
Since this projection is continuous and $C$ is connected, $\Phi(K)$ is connected.

The statement for $K'$ and $\Psi$ follows by symmetry.
\end{proof}

\begin{theorem}
\label{Thm:U<=5I}
Given any two Reeb graphs $(F,f)$ and $(G,g)$, we have
\[\dU(F,G) \leq 5 \dI(F,G).\]
\end{theorem}

\begin{proof}
Let $(\phi,\psi)$ be a $\delta$-interleaving of $(F,f)$ and $(G,g)$, so to any $x\in F$, there is associated a subset $\Phi(x)\subseteq G$ that is a connected component of $g^{-1}[f(x)\pm\delta]$.
Similarly, for any $y\in G$, $\Psi(y)$ is a connected component of $f^{-1}[g(y)\pm\delta]$.
We construct a subspace $Z \subseteq F \times G$ and two functions $\hat f, \hat g \colon Z \to \R$ with $\|\hat f - \hat g\|_\infty \leq 5\delta$ such that the canonical projections $\pr_F \colon F \times G \to F$, $\pr_G \colon F \times G \to G$ restrict to Reeb quotient maps $q_F \colon Z \to F$ of $\hat f$ and $q_G \colon Z \to G$ of $\hat g$, proving that $\dU \leq 5\dI$.

For $x\in F$ and $y\in G$, let
\begin{align*}
C(x) &= K_x(f^{-1}[f(x)\pm2\delta]), &
C(y) &= K_y(g^{-1}[g(y)\pm2\delta]),
\end{align*}
and let
\[
Z = \bigcup_{x \in F} C(x) \times \Phi(C(x)) \cup \bigcup_{y \in G} \Psi(C(y)) \times C(y)
\subseteq F \times G.
\]
To show that $\|\hat f - \hat g\|_\infty \leq 5\delta$, by symmetry it suffices to show that for every $x \in F$ and every $(z,y) \in C(x) \times \Phi(C(x))$ we have $|f(z)-g(y)| \leq 5\delta$.
Pick $w\in C(x)$ such that $y\in \Phi(w)$. We have $|f(z)-f(w)|\leq 4\delta$ by construction of $C(x)$, and $|f(w)-g(y)|\leq \delta$ by definition of $\delta$-interleaving.
Together, we have $|f(z)-g(y)| \leq 5\delta$ as claimed.

To show that $q_F\colon Z\to F$ is surjective, simply observe that for any $x\in F$ and $y\in \Phi(x)$,
\[
(x,y)\in C(x)\times \Phi(C(x))\subseteq Z.
\]
A similar argument shows that also $q_G\colon Z\to G$ is surjective.

It remains to show that the fibers of $q_F$ are connected; by symmetry, the same is then true for $q_G$ as well.
The fiber of $z \in F$ is of the form $q_F^{-1}(z) = \{z\} \times G_z \subseteq Z$, where $G_z = q_G(q_F^{-1}(z)) \subseteq G$ is a subspace, homeomorphic to the fiber.
Note that $G_z$ has the explicit description
\[
G_z = \bigcup_{\substack{x \in F\\z \in C(x)}} \Phi(C(x)) \cup \bigcup_{\substack{y \in G\\z \in \Psi(C(y))}} C(y).
\]
Clearly, $\Phi(z)$ is contained in any $\Phi(C(x))$ with $x \in F$ and $z \in C(x)$.
In addition, $\Phi(z)\subseteq C(y)$ for any $y\in G$ such that $z\in \Psi(y)$, as $\Phi(\Psi(y))\subseteq C(y)$ by definition of interleaving.
By \cref{Lem:Phi_preserves_conn}, $\Phi(C(x))$ is connected. Thus,
\[
G'_z = \bigcup_{\substack{x \in F\\z \in C(x)}} \Phi(C(x)) \cup \bigcup_{\substack{y \in G\\z \in \Psi(y)}} C(y)\subseteq G_z
\]
is connected, since it is a union of connected sets that all contain $\Psi(z)$. To complete the proof that $G_z$ is connected, it suffices to show that $C(y)$ intersects $G'_z$ for all $y\in G$ such that $z \in \Psi(C(y))$.
To see this, observe that there is a $w\in C(y)$ such that $z\in \Psi(w)$, and $w\in C(w)\subseteq G'_z$.
\end{proof}

\subsection{Relating functional contortion and interleaving distance}
\label{subsec:FC<=3I}

\begin{theorem}
\label{Thm:FC<=3I}
Given any two Reeb graphs $(F,f)$ and $(G,g)$, we have
\[\dFC(F,G) \leq 3 \dI(F,G).\]
\end{theorem}

\begin{proof}
Let $(\phi,\psi)$ be a $\delta$-interleaving of $(F,f)$ and $(G,g)$ for some $\delta\geq 0$.
We will show that for an arbitrary $\epsilon>0$, there are $\mu\colon F\to G$ and $\nu\colon G\to F$ with functional contortion $3\delta+3\epsilon$. %

Pick a discrete subset $S\subset F$ containing all the $0$-cells of $F$ such that for each $1$-cell $C$ and connected component $I$ of $C\setminus S$, the interval $f(I)$ has length less than $\epsilon$.
Pick a discrete subset $T\subset G$ with the same properties.

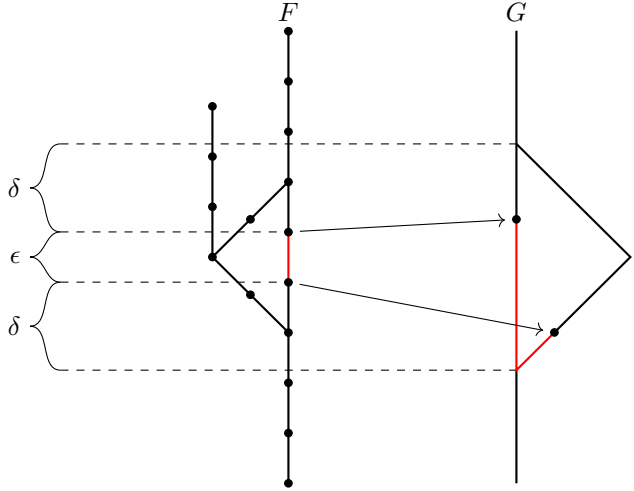
\begin{figure}
\begin{tikzpicture}[scale=1]
\draw (-2,0.333) to[out=180,in=0] (-2.4,0.917) to[out=0,in=180] (-2,1.5);
\node[left] at (-2.4,0.917){$\delta$};
\draw (-2,-0.333) to[out=180,in=0] (-2.4,-0.917) to[out=0,in=180] (-2,-1.5);
\node[left] at (-2.4,-0.917){$\delta$};
\draw (-2,-0.333) to[out=180,in=0] (-2.4,0) to[out=0,in=180] (-2,0.333);
\node[left] at (-2.4,0){$\epsilon$};
\draw[thick] (1,-3) to (1,-0.333);
\draw[thick,color=red] (1,-0.333) to (1,0.333);
\draw[thick] (1,0.333) to (1,3);
\draw[thick] (1,-1) to (0,0) to (1,1);
\draw[thick] (0,0) to (0,2);
\node[above] at (1,3){$F$};
\foreach \Y in {-3,-2.334,...,3}
\draw[color=black,fill=black] (1,\Y) circle (.05);
\draw[color=black,fill=black] (.5,-.5) circle (.05);
\draw[color=black,fill=black] (.5,.5) circle (.05);
\foreach \Y in {0,0.666,...,2}
\draw[color=black,fill=black] (0,\Y) circle (.05);
\draw[dashed] (-2,0.333) to (1,0.333);
\draw[dashed] (-2,-0.333) to (1,-0.333);
\draw[dashed] (-2,1.5) to (4,1.5);
\draw[dashed] (-2,-1.5) to (4,-1.5);
\coordinate (x) at (4,.5){};
\coordinate (y) at (4.5,-1){};
\draw[thick] (4,-3) to (4,-1.5);
\draw[thick] (x) to (4,3);
\draw[thick] (y) to (5.5,0) to (4,1.5);
\draw[thick,color=red] (y) to (4,-1.5) to (x);
\node[above] at (4,3){$G$};
\draw[->,shorten >=.15cm,shorten <=.15cm] (1,0.333) to (x);
\draw[color=black,fill=black] (x) circle (.05);
\draw[->,shorten >=.15cm,shorten <=.15cm] (1,-0.333) to (y);
\draw[color=black,fill=black] (y) circle (.05);
\end{tikzpicture}
\caption{Construction of a functional contortion pair for \cref{Thm:FC<=3I}. The points of $S$ are shown as black dots. The arrows and the red segments in $F$ and $G$ show $\mu$ applied to two points in $S$ and how we can extend $\mu$ to the segment between the points.}
\label{Fig:sampling_of_graph}
\end{figure}

For each $x\in S$, pick an arbitrary $y\in \Phi(x)$ and define $\mu(x) = y$. Similarly, for each $y\in T$, let $\nu(y)\in \Psi(y)$.
Let $I$ be a connected component of $F\setminus S$. Observe that $I$ is contained in a single $1$-cell $C$, and that its closure $\overline I$ contains %
two points $z,z'\in S$. By our assumptions on $I$, $\overline I$ is contained in a connected component of $f^{-1}[a,b]$ for some $a<b$ with $b-a=\epsilon$.
By \cref{Lem:Phi_preserves_conn}, it follows that $\Phi(\overline I)$ is contained in a connected component $K$ of $g^{-1}[a-\delta,b+\delta]$.
We can therefore 
use a a path from $\mu(z)$ to $\mu(z')$ in $K$ to extend $\mu$ continuously to $\overline I$; see \cref{Fig:sampling_of_graph}.
This implies that for all $x\in \overline I$, $\mu(x)$ and $\Phi(x)$ are both contained in $K$ and thus also in the same component of $g^{-1}[f(x)\pm(\delta+\epsilon)]$, as
\[
g^{-1}[a-\delta, b+\delta] \subseteq g^{-1}[f(x)\pm(\delta+\epsilon)].
\]
We do the same for all connected components of $F\setminus S$ and define $\nu$ similarly on $G\setminus T$.

Let $\delta'=\delta+\epsilon$. We now prove that $(\mu,\nu)$ is a $3\delta'$-contortion pair. By symmetry, it is enough to show that for any $x\in F$, $x$ and $\nu(\mu(x))$ are connected in $f^{-1}[g(\mu(x))\pm3\delta']$.
The $\delta$-interleaving $(\Phi,\Psi)$ induces a $\delta'$-interleaving $(\Phi',\Psi')$ canonically:
For $x\in F$, $\Phi'(x)$ is the connected component of $g^{-1}[f(x)\pm3\delta']$ containing $\Phi(x)$ as a subset, and $\Psi'(y)$ is defined similarly for $y\in G$.
We observed that $\mu(x)$ and $\Phi(x)$ are connected in $g^{-1}[f(x)\pm\delta']$, so $\mu(x)\in \Phi'(x)$.
Similarly, $\nu(\mu(x))\in \Psi'(\mu(x))$. Putting the two together, we get $\nu(\mu(x))\in \Psi'(\Phi'(x))$.
By definition of interleaving, we have $\Psi'(\Phi'(x)) \subseteq K_x(f^{-1}[f(x)\pm2\delta'])$. Since $|f(x)-g(\mu(x))|\leq \delta'$, we have
\[
f^{-1}[f(x)\pm2\delta'] \subseteq f^{-1}[g(\mu(x))\pm3\delta'],
\]
so
\[
\nu(\mu(x))\in \Psi'(\Phi'(x)) \subseteq K_x(f^{-1}[g(\mu(x))\pm3\delta']),
\]
which is what we wanted to prove.
\end{proof}

\section{Tightness of the bi-Lipschitz bounds for Reeb graphs}
\label{sec:lower}

In this section, we prove tightness of all the bounds stated in \cref{thm:main}.
For each of the inequalities, we  give examples of Reeb graphs for which the bounds are attained either tightly or approximately up to an arbitrarily small constant.
Specifically, if $d$ and $d'$ are distances on Reeb graphs such that $d\leq Cd'$, we either construct Reeb graphs $F,G$ such that $d(F,G) = C$ and $d'(F,G) = 1$, or we construct families $F_n,G_n$ such that $\lim_{n\to \infty} d(F_n,G_n) = C$ and $ \lim_{n\to \infty} d'(F_n,G_n) = 1$.

In \cref{sec:FCFD} and \cref{sec:FDI} we construct Reeb graphs with branches that extend to infinity.
This is just a technical convenience; the distance bounds still hold if we restrict the graphs to a large enough finite interval.
Thus, all the bounds in \cref{thm:main} are tight also in the case where we require the graphs to be finite CW complexes supported on a finite interval.

\subsection{Relating functional contortion and functional distortion distance}
\label{sec:FCFD}

We prove tightness of the bound $\dFC \leq 3\dFD$ from \cref{thm:main}
by giving an example of Reeb graphs $(F,f)$ and $(G,g)$ such that $\dFC(F,G) \geq 3$ and $\dFD(F,G) \leq 1$.

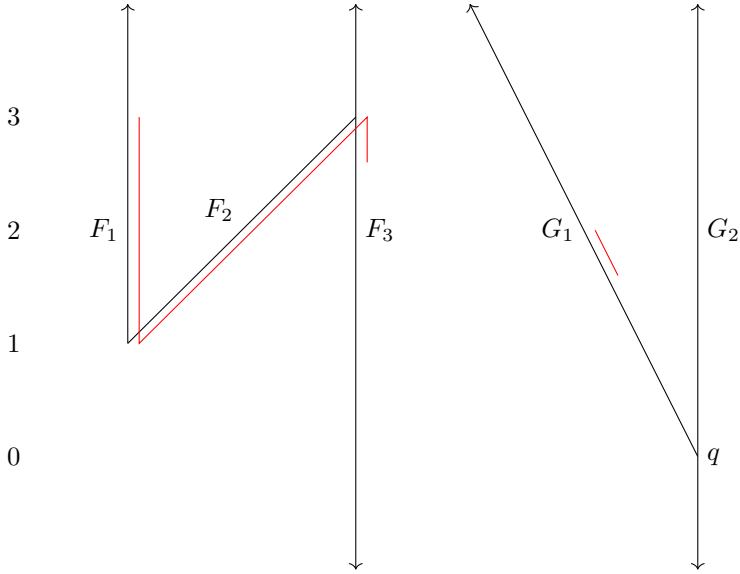
\begin{figure}
\begin{tikzpicture}[scale=1.5]
\begin{scope}
\draw[->] (0,1) to (0,4);
\draw[<->] (2,4) to (2,-1);
\draw (0,1) to (2,3);
\draw[color=red] (2.1,2.6) to (2.1,3) to (.1,1) to (.1,3);
\node[left] at (0,2){$F_1$};
\node[left,above] at (.8,2){$F_2$};
\node[right] at (2,2){$F_3$};
\node at (-1,0){$0$};
\node at (-1,1){$1$};
\node at (-1,2){$2$};
\node at (-1,3){$3$};
\end{scope}
\begin{scope}[xshift=3cm]
\draw[->] (2,0) to (0,4);
\draw[<->] (2,-1) to (2,4);
\draw[color=red] (1.1,2) to (1.3,1.6);
\node[right] at (2,0){$q$};
\node[left] at (1,2){$G_1$};
\node[right] at (2,2){$G_2$};
\end{scope}
\end{tikzpicture}
\caption{The Reeb graphs $F$ (left) and $G$ (right) of \cref{sec:FCFD}.
The red parts of $G$ and $F$ show $\{G_1\} \times [2-\delta,2]$ and a possible choice for $\psi(\{G_1\}\times [2-\delta,2])$, respectively, where $\psi$ is one of the $(1+\delta)$-distortion maps.}
\label{fig_FC_FD}
\end{figure}
We construct Reeb graphs by specifying a finite set of line or line segments of the form $\{l\}\times I$, where $I$ is an interval and $l$ works as a distinct label, and gluing them together with a finite set of pairwise identifications. 
We use the label $l$ to refer to the corresponding piece of the Reeb graph.
For instance, the graph $F$ to the left in \cref{fig_FC_FD} can be defined formally as
\[
(\{F_1\}\times [1,\infty) \cup \{F_2\}\times [1,3] \cup \{F_3\}\times (-\infty,\infty))/\sim,
\]
where $\sim$ is the equivalence relation generated by $(F_1,1)\sim(F_2,1)$ and $(F_2,3)\sim(F_3,3)$.
We occasionally abuse notation and do not distinguish between a point of an interval and its equivalence class in the quotient.
We will also refer to subgraphs by the first coordinate of their points, for instance by writing $F_1$ instead of $\{F_1\}\times [1,\infty)$. 
The resulting quotient space is a $1$-dimensional CW complex (the infinite edges can be subdivided into an infinite sequence of $1$-cells), and as long as we only glue together points with equal second coordinates, we get a well-defined Reeb graph map $F \to \R$ by projection onto the second coordinate. Similarly, we define
\[
G = (\{G_1\}\times [0,\infty) \cup \{G_2\}\times (-\infty,\infty))/\sim,
\]
with $(G_1,0)\sim(G_2,0)$. Note that both $F$ and $G$ are contour trees.

We claim that $\dFD(F,G)\leq1$ and $\dFC(F,G)\geq3$. To show the former, we construct a $(1+\delta)$-distortion pair $(\phi,\psi)$ for an arbitrarily small $\delta>0$.
Let $\phi$ be defined by
\begin{align*}
\phi(F_1,t) &= (G_1,t-1)\\
\phi(F_2,t) &= (G_2,t-1)\\
\phi(F_3,t) &= (G_2,t-1)
\end{align*}
for all $t\in\R$ whenever defined, and $\psi$ by
\begin{align*}
\psi(G_2,t) &= (F_3,t+1), t \in \R\\
\psi(G_1,t) &= (F_1,t+1), t \geq 2\\
\psi(G_1,t) &= (F_3,t+1), 0\leq t \leq 2-\delta,
\end{align*}
and we extend $\psi$ on $\{G_1\}\times[2-\delta,2]$ in such a way that $\psi$ is continuous, $\psi(\{G_1\}\times[2-\delta,2]) \subseteq f^{-1}([1,3])$, and $\psi$ agrees with the above
on $\{G_1\}\times \{2-\delta,2\}$.

First of all, it is straightforward to check that $||f-g\circ\phi||_\infty,||g-f\circ\psi||_\infty\leq 1+\delta$.
For $x\in F$, $y\in G$, consider the points in correspondence with $x$ and $y$, respectively:
\begin{align*}
C(x) &= \{z\in G \mid (x,z)\in C(\phi,\psi)\},\\
C(y) &= \{z\in F \mid (z,y)\in C(\phi,\psi)\}.
\end{align*}
Furthermore, let $T(x) =B(C(x))$ be the connecting subtree of $C(x)$ (see \cref{def:B(X)}), and define $T(y)$ analogously. 
For $x\in f^{-1}(\R \setminus [1,3])$, $T(x)$ is a single point $y$ with $g(y) = f(x) - 1$, and similarly, for $y\in g^{-1}(\R \setminus [0,2])$, $T(y)$ is a single point $x$ with $f(x) = g(y) + 1$.
Moreover, for $x\in f^{-1}[1,3]$ we have $C(x) \subseteq g^{-1}[0,2]$ and therefore $T(x) \subseteq g^{-1}[0,2]$ since $g^{-1}[0,2]$ is connected, and for $y\in g^{-1}[0,2]$ we have $T(y) \subseteq f^{-1}[1,3]$ for similar reasons.

Now let $(x,y), (x',y') \in C(\phi,\psi)$ be corresponding points.
We aim to show that given an interval $I=[a,b]$ such that $x$ and $x'$ are connected in $f^{-1}(I)$, there is an interval $J$ such that $y$ and $y'$ are connected in $g^{-1}(J)$ and the length of $J$ exceeds the length of $I$ by at most $2$.
We claim that $J$ can be chosen as an interval in the union $[a-1,b-1] \cup [0,2]$.
Since any such interval has length at most $b-a+2$, the claim follows.
Let $K$ be the connected component of both $x$ and $x'$ in $f^{-1}(I)$.
Note that $f(K) \subseteq [a,b]$.
Now consider the subset $L=\bigcup_{x'' \in K}T(x'') \subseteq G$.
This set is connected since it contains the connected subset $\phi(K)$, which intersects every $T(x'')$ for $x'' \in K$, and $T(x'')$ is also connected.
Furthermore, $L$ contains both $T(x)$ and $T(x')$, which in turn contain the points $y, y'$.
Finally, recall that we have either $g(T(x'')) = \{f(x'')-1\}$ or $g(T(x'')) \subseteq [0,2]$.
We conclude that
\[g(L) = \bigcup_{x'' \in K}g(T(x'')) \subseteq \bigcup_{x'' \in K}(f(x'')-1) \cup [0,2] \subseteq [a-1,b-1] \cup [0,2].\]

An analogous argument also shows that given an interval $J=[a,b]$ such that $y$ and $y'$ are connected in $g^{-1}(J)$, there is an interval $I \subseteq [a+1,b+1] \cup [1,3]$ such that $x$ and $x'$ are connected in $f^{-1}(I)$, so that the length of $I$ again exceeds the length of $J$ by at most $2$.

Now let $\epsilon<3$. Assume for a contradiction that $(\phi,\psi)$ is an $\epsilon$-contortion pair.
By the definition of $\epsilon$-contortion, $\psi(\phi(F_1,7))$ and $(F_1,7)$ are connected in
\[
f^{-1}[g(\phi(F_1,7))-\epsilon, g(\phi(F_1,7))+\epsilon]\subset f^{-1}[7 -2\epsilon, 7 +2\epsilon],
\]
which means that $\psi(\phi(F_1,7))$ lies in $F_1$.
Similarly, $\psi(\phi(F_3,7))$ lies in $F_3\cup F_2$.
Now, if $\phi(F_1,7)$ and $\phi(F_3,7)$ are connected in $g^{-1}[7-\epsilon,7+\epsilon]$, then by continuity of $\psi$, $\psi(\phi(F_1,7))$ and $\psi(\phi(F_3,7))$ are connected in $g^{-1}[7-2\epsilon,7+2\epsilon]$, a contradiction.
Thus, $\phi(F_1,7)$ and $\phi(F_3,7)$ lie in different branches of $G$; that is, $\phi(F_1,7)\in G_i$ and $\phi(F_3,7)\in G_j$, where either $i=1$ and $j=3$, or $i=3$ and $j=1$.
Let $B=B((F_1,7),(F_3,7))$, and let $B' = B(\phi(F_1,7),\phi(F_3,7))$.
Then $B'$ contains the point $q\coloneqq (G_1,0)$.
By continuity of $\phi$, $B'\subseteq \phi(B)$, so there must be a point $p \in B$ with $\phi(p)=q$.
Since $g(q) = 0$, we must have $f(p) \leq \epsilon < 3$, and thus $p \notin F_3$.
On the other hand, we must have $\psi(q) \subseteq \{F_3\} \times (-3,3)$, as the segment $\{G_2\} \times [-2,0]$ ending at $q$ must be mapped into a connected component of $f^{-1}[-2-\epsilon,\epsilon]$, and the point $(G_2,-2)$ can only be mapped to $g^{-1}[-2-\epsilon,-2+\epsilon]$, which is a subset of the connected component $\{F_3\} \times [-2-\epsilon,\epsilon]$.
Now $p \in \phi^{-1}(q)$ and $\psi(q)$ lie in different connected components of $f^{-1}[-\epsilon, \epsilon]$, and again we have a contradiction.
We conclude that there is no $\epsilon$-contortion pair between $F$ and $G$. 

\subsection{Relating functional distortion and interleaving distance}
\label{sec:FDI}

Next, we prove tightness of the bound $\dFD \leq 3\dI$ from \cref{thm:main}
by giving an example of Reeb graphs $F$ and $G$ such that $\dFD(F,G) \geq 3$ and $\dI(F,G)\leq 1$.

\begin{figure}
\begin{tikzpicture}[scale=.65]
\begin{scope}[xshift=-.5cm]
\node at (-1,-3){$-3$};
\node at (-1,-2){$-2$};
\node at (-1,-1){$-1$};
\node at (-1,0){$0$};
\node at (-1,1){$1$};
\node at (-1,2){$2$};
\node at (-1,3){$3$};
\end{scope}
\begin{scope}
\node at (1,4.5){$F$};
\draw[<->] (2,-4) to (2,4);
\draw[->] (0,1) to (0,4);
\draw (0,1) to (2,3);
\node[left] at (0,2.5){$F_1$};
\node[left,above] at (.8,2){$F_2$};
\node[right] at (2,0){$F_3$};
\draw[->] (0,-1) to (0,-4);
\draw (0,-1) to (2,-3);
\node[left] at (0,-2.5){$F_4$};
\node[left,below] at (.8,-2){$F_5$};
\end{scope}
\begin{scope}[xshift=3cm]
\node at (1,4.5){$G$};
\draw[<->] (0,-4) to (2,0) to (0,4);
\draw[<->] (2,-4) to (2,4);
\node[right] at (2,0){$q$};
\node[left] at (1,2){$G_1$};
\node[right] at (2,-1.5){$G_2$};
\node[left] at (1,-2){$G_3$};
\end{scope}
\begin{scope}[xshift=7cm]
\node at (1,4.5){$\mathcal{U}_1 F$};
\draw[<->] (2,-4) to (2,4);
\draw[->] (2,2) to (0,4);
\draw (0,0) to (2,2);
\draw[->] (2,-2) to (0,-4);
\draw (.5,0) to (2,-2);
\end{scope}
\begin{scope}[xshift=10cm]
\node at (1,4.5){$\mathcal{U}_1 G$};
\draw[<-] (0,-4) to (2,-1);
\draw[->] (2,1) to (0,4);
\draw[<->] (2,-4) to (2,4);
\end{scope}
\begin{scope}[xshift=14cm]
\node at (1,4.5){$\mathcal{U}_2 F$};
\draw[<->] (2,-4) to (2,4);
\draw[->] (2,3) to (1,4);
\draw (0,-1) to (2,1);
\draw[->] (2,-3) to (1,-4);
\draw (2,-1) to (1.1,-.1);
\draw (.9,.1) to (0,1);
\end{scope}
\begin{scope}[xshift=17cm]
\node at (1,4.5){$\mathcal{U}_2 G$};
\draw[<-] (0,-4) to (2,-2);
\draw[->] (2,2) to (0,4);
\draw[<->] (2,-4) to (2,4);
\end{scope}
\end{tikzpicture}
\caption{The Reeb graphs $F$ and $G$ from \cref{sec:FDI} and their $1$- and $2$-smoothings.}
\label{fig_Fd_I}
\end{figure}
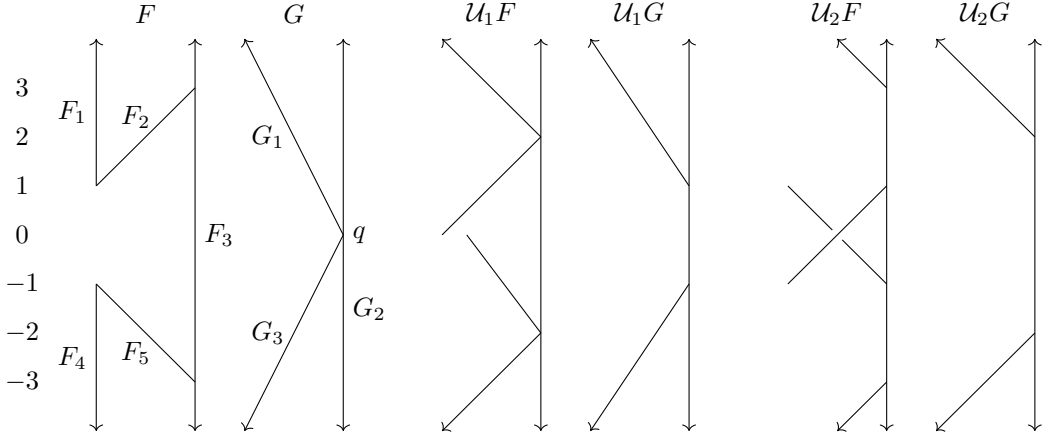
Informally, we let $(F,f)$ and $(G,g)$ be symmetric around level $0$, and let the upper parts be identical to those in the previous section; see \cref{fig_Fd_I}.
Formally, we define $(F,f)$ as
\[
\{F_1\}\times [1,\infty) \cup \{F_2\}\times [1,3] \cup \{F_3\}\times (-\infty,\infty) \cup \{F_4\}\times (-\infty,-1] \cup \{F_5\}\times [-3,-1],
\]
with the relations
\[
(F_1,1)\sim(F_2,1), (F_2,3)\sim(F_3,3), (F_4,-1)\sim(F_5,-1), (F_3,-3)\sim(F_5,-3);
\]
and $(G,g)$ as
\[
\{G_1\}\times [0,\infty) \cup \{G_2\}\times (-\infty,\infty) \cup \{G_3\}\times (-\infty,0],
\]
with the relations $(G_1,0)\sim (G_2,0)\sim (G_3,0)$.

Suppose we have an $\epsilon$-distortion pair $(\phi,\psi)$ between $F$ and $G$ for $\epsilon<3$.
Let $q=(G_2,0)$.
By arguments analogous to those in the previous subsection used to show non-existence of $\epsilon$-contortions for $\epsilon<3$, $\phi^{-1}(q)$ contains a point $p$ in $B((F_1,7),(F_3,7))$, and similarly, $\phi^{-1}(q)$ also contains a point $p'$ in $B((F_4,-7),(F_3,-7))$.
Any path between $p$ and $p'$ contains $\{F_3\}\times [-3,-3]$.
This contradicts the assumption that $(\phi,\psi)$ is an $\epsilon$-distortion pair for $\epsilon<3$.

Next we show that $\dI(F,G) \leq 1$ by describing a $1$-interleaving between $F$ and $G$.
By the remarks after \cref{def:interleaving}, it suffices to construct level-preserving functions $\mu\colon F\to \mathcal{U}_1 G$ and $\nu\colon G\to \mathcal{U}_1 F$ such that $\nu_1\circ\mu = \zeta_F^2$ and $\mu_1\circ\nu = \zeta_G^2$.
The smoothed Reeb graphs $\mathcal{U}_1 F$ and $\mathcal{U}_1 G$ are illustrated in \cref{fig_Fd_I}.
We define $\mu$ and $\nu$ by letting
\begin{align*}
\mu(F_1,4) &= \zeta_G^1(G_1,4), &
\mu(F_3,4) &= \zeta_G^1(G_2,4),\\
\mu(F_4,-4) &= \zeta_G^1(G_3,-4), &
\mu(F_3,-4) &= \zeta_G^1(G_2,-4),\\
\nu(G_1,4) &= \zeta_F^1(F_1,4), &
\nu(G_2,4) &= \zeta_F^1(F_3,4),\\
\nu(G_3,-4) &= \zeta_F^1(F_4,-4), &
\nu(G_2,-4) &= \zeta_F^1(F_3,-4).
\end{align*}
We leave to the reader to verify that this determines $\mu$ and $\nu$ uniquely (under the assumptions that the functions are level-preserving and continuous).
It follows that $\nu_1\circ\mu(p) = \zeta_F^2(p)$ for $p\in \{(F_1,4), (F_3,4), (F_4,-4), (F_3,-4)\}$; for instance
\[
\nu_1(\mu(F_1,4)) = \nu_1(\zeta_G^1(G_1,4)) = \zeta_{F_1}^1(\nu(G_1,4)) = \zeta_{F_1}^1(\zeta_F^1(F_1,4)) = \zeta_F^2(F_1,4),
\]
where the second equality follows from the definition of the induced map $\nu_1$ given after \cref{def:interleaving}.
These four values determine $\nu\circ\mu$ uniquely, so we must have $\nu_1\circ\mu = \zeta_F^2$.
A similar argument shows that $\mu_1\circ\nu = \zeta_G^2$, and thus $(\mu,\nu)$ is a $1$-interleaving.

\subsection{Relating universal and functional contortion distance}

We now prove tightness of the bound $\dU \leq 3\dFC$ from \cref{thm:main}
by giving an example family of Reeb graphs $F_n$ and $G_n$ such that $\dU(F_n,G_n) \geq 3$ and $\dFC(F_n,G_n)\leq 1+\frac{1}{2n}$.
Many of the ideas appearing in this subsection will be recycled in the next, where we prove tightness of the bound $\dU\leq 5\dI$.
In both subsections, we will use graphs of the form $F_{n,h}$ and $G_{n,h}$, where $n$ is a positive integer, and use the fact that $\dU(F_{n,h},G_{n,h}) \geq h$ (for $h=3$ and $h=5$, respectively), which we prove in this subsection.
In \cref{fig_fins1}, we have shown $F_{8,3}$ on the left and $G_{8,3}$ on the right.

Recall from \cref{sec:FCFD} how we construct Reeb graphs by gluing together labeled intervals.
$F_{n,h}$ contains the following pieces, where $r_i\coloneqq \frac{n-i}{n}(h+1)$, $s_i \coloneqq r_i+h-1$ and $t_i\coloneqq (h-1)(n-i) +4h$ (suppressing the dependence on $n$ and $h$ from the notation):
\begin{align*}
&\{L, R \}\times [0,h+1],\\
&\{L_i,R_i\}\times [r_i,s_i] \text{ and }
\{U_i\}\times [s_i,t_i] \text{ for } i=0,1, \dots, n
\end{align*}
that are glued together by the following relations
\begin{align*}
(L,0) &\sim (R,0)\\
(L,r_i) &\sim (L_i,r_i)\\
(R,r_i) &\sim (R_i,r_i)\\
(L_i,s_i)&\sim (R_i,s_i)\sim (U_i,s_i)
\end{align*}
whenever they make sense.

$G_{n,h}$ contains the pieces
\begin{align*}
&\{M\}\times [-1,h],\\
&\{O_i\}\times [r_i-1,t_i-1] \text{ for } i=0, \dots, n
\end{align*}
with relations $(M,r_i-1)\sim (O_i,r_i-1)$.

\begin{figure}
\begin{tikzpicture}[scale=.7]
\begin{scope}[xscale=2]
\foreach \i in {-1,...,6}{
\node at (-1,\i){$\i$};
}
\begin{scope}[xshift=1cm]
\draw (0,4) to (0,0) to (2,0) to (2,4);
\foreach \i in {0,.5,...,3.5}{
\draw (0,\i) to (1,\i+2) to (2,\i);
\draw[->] (1,\i+2) to (1,\i+2.3);
}
\draw (0,4) to (1,6) to (2,4);
\draw[->] (1,6) to (1,7);
\node[right] at (1,6){$u_0$};
\node[right] at (1.2,5.5){$u_1$};
\node[left] at (1,6.5){$U_0$};
\node[left] at (0,2){$L$};
\node at (.3,5.3){$L_0$};
\node at (1.3,.9){$R_8$};
\node[right] at (2,2){$R$};
\end{scope}
\begin{scope}[xshift=5cm]
\draw (0,-1) to (0,3);
\foreach \i in {-1,-.5,...,3}{
\draw[->] (0,\i) to (1,\i+2);
}
\node[right] at (1,5){$O_0$};
\node[right] at (1,1){$O_8$};
\node[left] at (0,1){$M$};
\node[left] at (0,3){$m$};
\end{scope}
\end{scope}
\draw [fill=black] (4,6) circle (.07);
\draw [fill=black] (4,5.5) circle (.07);
\draw [fill=black] (10,3) circle (.07);
\end{tikzpicture}
\caption{$F_{8,3}$ to the left, $G_{8,3}$ to the right.
The horizontal segment in $F_{8,3}$ represents a single point.
The arrows illustrate branches that go further up than what we have drawn.}
\label{fig_fins1}
\end{figure}
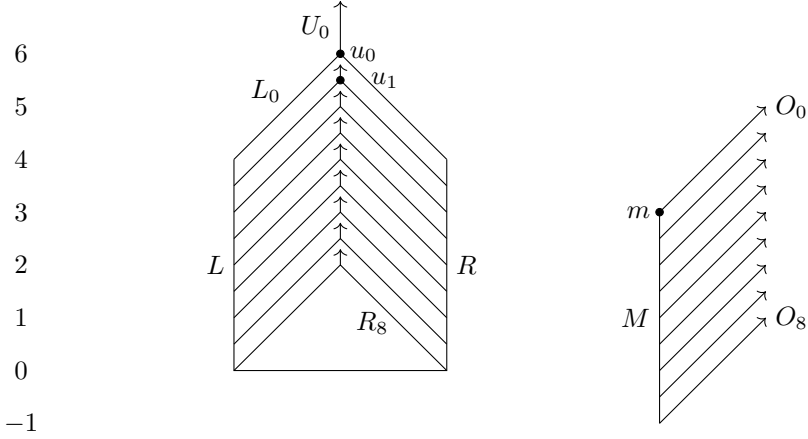

We first prove $\dU(F_{n,h},G_{n,h}) \geq h$, then $\dFC(F_{n,3},G_{n,3})\leq 1+\frac{1}{2n}$.
The first of these inequalities will be used in both this subsection and the next; here we need the inequality with $h=3$, and in the next subsection we need it with $h=5$.

Let $G=G_{n,h}$ and $F=F_{n,h}$.
Suppose $\dU(F,G)<h$.
Then there is an $\epsilon<h$, a space $X$ and maps $f:X\to \R$ and $g:X\to \R$ such that $F$ and $G$ are the induced Reeb graphs of $f$ and $g$ (up to isomorphism), respectively, and $\|f-g\|_\infty\leq \epsilon$.
We denote the induced maps $X\to F$ and $X\to G$ by $\pi_F$ and $\pi_G$. Let $\alpha= \pi_G\circ \pi_F^{-1}:F \to \Pp(G)$, where $\Pp(G)$ denotes the power set of $G$.
Note that for $p\in F$, $\pi_F^{-1}(p)$ is nonempty and connected by the definition of induced Reeb graph, and $\pi_G$ preserves connectedness since it is continuous.
Thus, $\alpha(p)$ is a nonempty connected subset of $G$.
Also observe that since $\|f-g\|_\infty\leq \epsilon<h$, we have $\alpha(p)\subseteq g^{-1}(f(p)-h,f(p)+h)$.
Define $\beta\colon G \to \Pp(F)$ symmetrically.
By construction, $q\in \alpha(p)$ if and only if $p\in \beta(q)$.
More concisely, $\alpha = \beta^{-1}$ and $\beta = \alpha^{-1}$.

We first show that $\alpha(U_i)\subseteq O_i$ by induction on $i$.
For $i=0$, consider the point $u_0=(U_0,t_0)$.
We have $\alpha(u_0)\subseteq g^{-1}(t_0-h, t_0+h)$, which is contained in $O_0$.
We have $\alpha(U_0)\subseteq g^{-1}[s_0-\epsilon, t_0+\epsilon]$ and $s_0-\epsilon> r_0-1$, so since $\alpha$ preserves connectivity, it follows that $\alpha(U_0) \subset O_0$.
For any $q\in \alpha(U_0)$, $\beta(q)$, which as observed is equal to $\alpha^{-1}(q)$, thus intersects $U_0$.
Picking $q\in g^{-1}(t_0-h,t_0+h)$ and using that $\beta$ preserves connectivity, we get $\beta(q)\subseteq U_0$ and thus also $\beta(\{O_0\}\times [s_0+h,t_0-1])\subseteq U_0$.

Assume by induction that $\alpha(U_j)\subseteq O_j$ and $\beta(\{O_j\}\times [s_j+h,t_j-1])\subseteq U_j$ for $j<i$.
Then, letting $u_i= (U_i,t_i)$, $\alpha(u_i)$ cannot intersect $O_j$ for any $j<i$, since this would contradict $\beta(\{O_j\}\times [s_j+h,t_j-1])\subseteq U_j$.
Thus, by an argument similar to the case $i=0$, we get $\alpha(U_i)\subseteq O_i$ and $\beta(O_i,x)\subseteq U_i$ for $x\geq s_i+h$.
This concludes the proof by induction.

In particular, $\alpha(U_i)\subseteq O_i$ implies that $u_i \notin \beta(M)$ for all $i$.
Let $m=(M,h)\in G$.
We get that $\beta(m)$ is contained in $f^{-1}[h-\epsilon, h+\epsilon]\setminus \{u_0, \dots, u_n\}$, in which the ``left side'' of $F$ is not connected to the ``right side''.
More precisely, since $\beta(m)$ is connected, we can in particular conclude that $\beta(m)$ does not intersect both $L_{0,1}\coloneqq L\cup L_0 \cup L_1$, and $R_{0,1}\coloneqq R\cup R_0 \cup R_1$.
Suppose without loss of generality that $\beta(m) \cap L_{0,1} = \emptyset$, or equivalently, $m\notin \alpha(L_{0,1})$.
But $\alpha(L_{0,1})$ contains both $\alpha(u_0)$ and $\alpha(u_1)$, so $\alpha(L_{0,1})$ intersects both $O_0$ and $O_1$.
Since $\alpha(L_{0,1})$ is connected, it must contain $m$, which is a contradiction.
We conclude that $\dU(F,G)\geq h$.

Next, let $h=3$, so $F = F_{n,3}$ and $G = G_{n,3}$.
We construct $\phi\colon F\to G$ and $\psi\colon G\to F$ with contortion $1+\frac{2}{n}$.
\begin{itemize}
	\item For $x\in [-1,3]$, let $\psi(M,x)= (L,x+1)$,
	\item for all $x$ and $i$ such that $(O_i,x)$ is defined, let $\psi(O_i,x)$ be $(L_i,x+1)$ if defined, and otherwise $(U_i,x+1)$.
\end{itemize}
For any $p$ in the image of $\psi$, let $\phi(p)$ be the unique point in $\psi^{-1}(p)$.
Given the left-right symmetry of $F$, there is an obvious level-preserving homeomorphism $\rho\colon F\to F$ flipping the left and right sides of $F$.
For any $p\in F$ for which we have not defined $\phi$ yet, we define $\phi(p)=\phi(\rho(p))$.

By construction, $\phi\circ\psi$ is the identity on $G$, while for $p\in F$, $\phi(p)$ is either $p$ or $\rho(p)$.
Therefore, to finish the proof, we only need to check that for any $p\in F$, $p$ and $\rho(p)$ are connected in $K\coloneqq f^{-1}[g(\phi(p))-1-\frac{2}{n},g(\phi(p))+1+\frac{2}{n}]$.
A careful look at \cref{fig_fins1} will reveal the main idea: for any interval $I\subseteq [0,6]$ of length at least $2+\frac{4}{n}$, $f^{-1}(I)$ contains $L_i\cup R_i$ for some $i$, allowing a path between $L$ and $R$ inside of $f^{-1}(I)$.
The rest of the proof is just a technical verification that this idea does indeed give us what we want.

First observe that $p$ and thus $\rho(p)$ are themselves contained in $K$.
If $p=\rho(p)$, we are done, so assume $p\neq \rho(p)$.
We can assume $p\in R$ or $p\in R_i$ for some $i$.
Note that it is enough to find a path contained in $K$ from $p$ to $(R,0)$ or $p_i$ for some $i$, as then there is a symmetric path in $K$ from $p'$ to the same point.
There is a monotone path downwards from $p$ to a point $p_\downarrow$ such that $f(p_\downarrow)=\max\{g(\phi(p))-1-\frac{2}{n},-5\}$.
The image of this path is unique and contained in $K$.
If $f(p_\downarrow)=-5$, then $p_\downarrow= (R,-5)$, and we are done.
Therefore, we can assume $f(p_\downarrow)=g(\phi(p))-1-\frac{2}{n}$.
By construction, $g(\phi(p))= f(p)-1$, so $g(\phi(p))-1-\frac{2}{n}\leq f(p)-2$.
Thus, $p_\downarrow$ lies on $R$, so there is a monotone path upwards from $p_\downarrow$ to one of the points $p_i$ with $f(p_i)\leq f(p_\downarrow)+2+\frac{4}{n}$, as the difference between the levels of the bottom points of $R_i$ and $R_{i+1}$ is $|r_i-r_{i+1}| = \frac{4}{n}$.
We have
\begin{align*}
f(p_i)\leq f(p_\downarrow)+2+\frac{4}{n} &= g(\phi(p))-1-\frac{2}{n} + 2+\frac{4}{n}\\
&= g(\phi(p))+1+\frac{2}{n},
\end{align*}
so this path is contained in $K$.
Thus, in all cases, $p$ and $\rho(p)$ are connected in $K$, and we can conclude that $(\phi, \psi)$ is a $\left(1+\frac{2}{n}\right)$-contortion pair.

\subsection{Relating universal and interleaving distance}
We now prove tightness of the bound $\dU \leq 5\dI$ from \cref{thm:main} by defining a family of Reeb graphs $F_n$ and $G_n$ such that $\dU(F_n,G_n) \geq 5$ and $\dI(F_n,G_n)\leq 1+\frac{3}{2n}$.
Specifically, we choose $F_n$ to be $F_{n,5}$ and $G_n$ to be $G_{n,5}$ as defined in the previous subsection.

We already showed that $\dU(F_{n,h},G_{n,h}) \geq h$, which gives us the first inequality by putting $h=5$.
To show that $\dI(F,G)\leq 1+\frac{3}{2n}$, we will apply the following lemma:

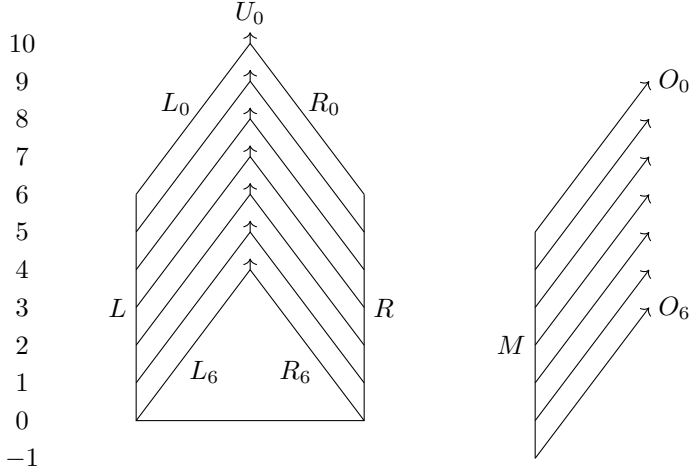
\begin{figure}
\begin{tikzpicture}[scale=.5]
\begin{scope}[xscale=1.5]
\foreach \i in {-1,...,10}{
\node at (-1,\i){$\i$};
}
\begin{scope}[xshift=1cm]
\draw (0,6) to (0,0) to (4,0) to (4,6);
\foreach \i in {0,...,6}{
\draw (0,\i) to (2,\i+4) to (4,\i);
\draw[->] (2,\i+4) to (2,\i+4.3);
}
\node[left] at (0,3){$L$};
\node at (.7,8.4){$L_0$};
\node at (1.2,1.3){$L_6$};
\node[right] at (4,3){$R$};
\node at (3.3,8.4){$R_0$};
\node at (2.8,1.3){$R_6$};
\node at (2,10.8){$U_0$};
\end{scope}
\begin{scope}[xshift=8cm]
\draw (0,-1) to (0,5);
\foreach \i in {-1,...,5}{
\draw[->] (0,\i) to (2,\i+4);
}
\node[left] at (0,2){$M$};
\node[right] at (2,9){$O_0$};
\node[right] at (2,3){$O_6$};
\end{scope}
\end{scope}
\end{tikzpicture}
\caption{$F_{6,5}$ to the left, $G_{6,5}$ to the right.
The arrows illustrate branches that go further up than what we have drawn.}
\label{fig_fins2}
\end{figure}

\begin{lemma}
\label{lem:induced_Int}
Let $(F,f)$ and $(G,g)$ be Reeb graphs and $\phi\colon F\to G$ and $\psi\colon G\to F$ be continuous functions with $\|f-g\circ\phi \|_\infty\leq \epsilon$ and $\|g-f\circ\psi \|_\infty\leq \epsilon$, and such that for all $p\in F$ and $q\in G$,
\begin{itemize}
\item $p$ and $\psi(\phi(p))$ are connected in $f^{-1}[f(p)\pm2\epsilon]$,
\item $q$ and $\phi(\psi(q))$ are connected in $g^{-1}[g(q)\pm2\epsilon]$.
\end{itemize}
Then there is an $\epsilon$-interleaving between $(F,f)$ and $(G,g)$.
\end{lemma}
\begin{proof}
Since $\|f-g\circ\phi \|_\infty\leq \epsilon$, $\phi$ induces a continuous level-preserving function $\bar{\phi}\colon F\to \Uu_\epsilon G$ by $\bar{\phi}(p) = q_{U_\epsilon G}(\phi(p),f(p)-g(\phi(p)))$.
Similarly, $\psi$ induces a continuous level-preserving function $\bar{\psi}\colon G\to \Uu_\epsilon F$.
As explained before \cref{def:interleaving}, $\bar \phi$ and $\bar \psi$ induce set-valued maps $\Phi\colon F\to \Pp(G)$ and $\Psi\colon G\to \Pp(F)$, and by construction, these satisfy $\phi(p)\in\Phi(p)$ and $\psi(q)\in\Psi(q)$ for all $p\in F$ and $q\in G$.

Since for all $p\in F$, $\Psi(\Phi(p))$ is connected and $p$ and $\psi(\phi(p))$ are connected in $f^{-1}[f(p)-2\epsilon, f(p)+2\epsilon]$ by the assumption in the lemma, and the same holds for $\Phi(\Psi(q))$ for all $q\in G$, we get that $\bar \phi$ and $\bar \psi$ form an $\epsilon$-interleaving.
\end{proof}

For simplicity, we write $F$ and $G$ for $F_n$ and $G_n$, respectively, from now on.
We construct $\phi\colon F\to G$ and $\psi\colon G\to F$ in a way entirely analogous to the morphisms with contortion $1+\frac{2}{n}$ in the previous subsection.
Specifically,
\begin{itemize}
\item For $x\in [-1,5]$, let $\psi(M,x)= (L,x+1)$,
\item for all $x$ and $i$ such that $(O_i,x)$ is defined, let $\psi(O_i,x)$ be $(L_i,x+1)$ if defined, and otherwise $(U_i,x+1)$.
\end{itemize}
For any $p$ in the image of $\psi$, let $\phi(p)$ be the unique point in $\psi^{-1}(p)$.
As in the previous subsection, there is a level-preserving homeomorphism $\rho\colon F\to F$ flipping the left and right sides of $F$.
For any $p\in F$ for which we have not defined $\phi$ yet, we define $\phi(p)=\phi(\rho(p))$.

We have that $\phi$ and $\psi$ satisfy $\|f-g\circ\phi \|_\infty\leq 1$ and $\|g-f\circ\psi \|_\infty\leq 1$.
$\phi\circ \psi$ is the identity on $G$, and for all $p\in F$, we have that $\psi(\phi(p))$ is either $p$ or $\rho(p)$.
Thus, to apply \cref{lem:induced_Int}, the only nontrivial condition that is left to check is that for all $p\in F$, $p$ and $\rho(p)$ are connected in $f^{-1}[f(p)-2-\frac{3}{n}, f(p)+2+\frac{3}{n}]$.
The argument is virtually the same as the proof of the $(1+\frac{2}{n})$-contortion property in the previous subsection; the details are left to the reader.
Thus, \cref{lem:induced_Int} tells us that there is a $\delta$-interleaving between $F$ and $G$, and we are done.

\section{Universality of the functional contortion distance for contour trees}
\label{sec:contour}

We now prove \cref{thm:contour}, which says that the universal and functional contortion distances are equal for contour trees.

\begin{proof}[Proof of \cref{thm:contour}]
We know $\dFC(F,G) \leq \dU(F,G)$ by \cref{Thm:FC<=U}, so it remains to prove $\dFC(F,G) \geq \dU(F,G)$.

Assume that there is a $\delta$-contortion pair $(\phi,\psi)$ between $F$ and $G$. %
We construct a subspace $Z \subseteq F \times G$ and two functions $\hat f, \hat g \colon Z \to \R$ with $\|\hat f - \hat g\|_\infty \leq \delta$ such that the canonical projections $\pr_F \colon F \times G \to F$, $\pr_G \colon F \times G \to G$ restrict to Reeb quotient maps $q_F \colon Z \to F$ of $\hat f$ and $q_G \colon Z \to G$ of $\hat g$, proving that $\dFC(F,G) \geq \dU(F,G)$.

Let $C(x)= \{\phi(x)\}\cup \psi^{-1}(x)$ for $x\in F$, and let $T(x) = B(C(x))$ be the connecting subtree of $C(x)$ (see \cref{def:B(X)}).
Let $T(y)$ be defined similarly for $y\in G$ switching $\phi$ and $\psi$.
Let $Z \subseteq F \times G$ be given by
\[Z = \left[\bigcup_{x\in F}\{x\}\times T(x)\right] \cup \left[\bigcup_{y\in G}T(y)\times \{y\}\right]. \]

To show that $\|\hat f - \hat g\|_\infty \leq \delta$, by symmetry it suffices to show that for every $x \in F$ and $y\in T(x)$,
$|f(x)-g(y')| \leq \delta$.
By definition of $\delta$-contortion pair, the connected component $K$ of $g^{-1}[f(x)-\delta,f(x)+\delta]$ containing $\phi(x)$ also contains $\psi^{-1}(x)$.
Thus, $T(x) \subseteq K$, and the statement follows.

For any $x\in F$, we have
\[
(x,\phi(x)) \in \{x\}\times T(x)\subseteq Z.
\]
It follows immediately that $q_F\colon Z\to F$ is surjective, and similarly for $q_G\colon Z\to G$.

It remains to show that the fibers of $q_F$ and $q_G$ are connected. By symmetry, we only need to prove this for $q_F$.
The fiber of $x \in F$ is of the form $q_F^{-1}(x) = \{x\} \times G_x \subseteq Z$, where $G_x = q_G(q_F^{-1}(x)) \subseteq G$ is a subspace, homeomorphic to the fiber.
To follow the arguments that follow, we suggest keeping an eye on \cref{fig:contour_proof}.
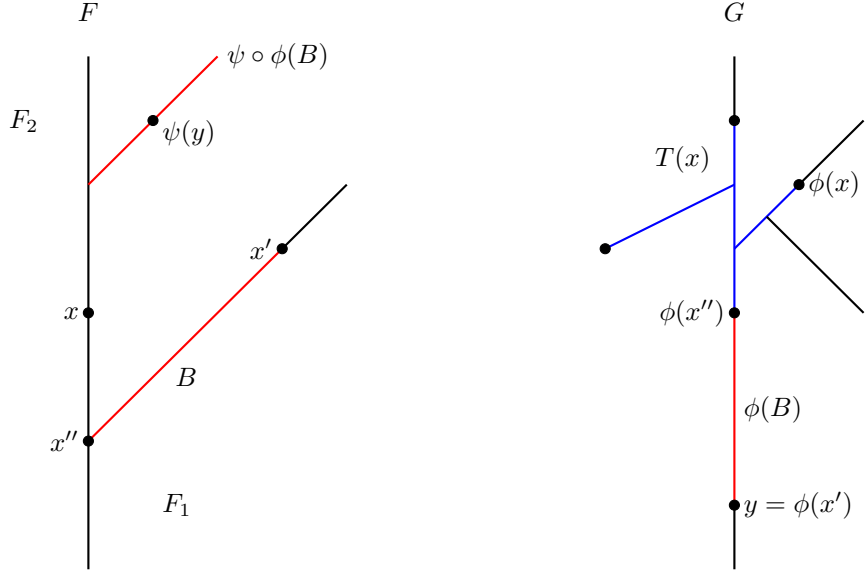
\begin{figure}
\centering
\begin{tikzpicture}[scale=1.7]
\begin{scope}
\draw[thick] (0,-2) to (0,2);
\draw[thick,color=red] (0,-1) to (1.5,.5);
\draw[thick] (1.5,.5) to (2,1);
\draw[thick,color=red] (0,1) to (1,2);
\node[above] at (0,2.2){$F$};
\node[left] at (0,0){$x$};
\node[left] at (1.5,.5){$x'$};
\node[left] at (0,-1){$x''$};
\node[right] at (.5,1.4){$\psi(y)$};
\node[right] at (1,2){$\psi\circ\phi(B)$};
\node[right] at (.6,-.5){$B$};
\node at (-.5,1.5){$F_2$};
\node[right] at (.5,-1.5){$F_1$};
\begin{scope}[every node/.style={draw,shape=circle,fill=black,scale=.4}]
\node at (0,0){};
\node at (1.5,.5){};
\node at (.5,1.5){};
\node at (0,-1){};
\end{scope}
\end{scope}
\begin{scope}[xshift=5cm]
\draw[thick,color=blue] (0,0) to (0,1.5);
\draw[thick] (0,1.5) to (0,2);
\draw[thick,color=blue] (0,.5) to (.5,1);
\draw[thick] (.5,1) to (1,1.5);
\draw[thick] (.25,.75) to (1,0);
\draw[thick,color=blue] (0,1) to (-1,.5);
\draw[thick] (0,-1.5) to (0,-2);
\draw[thick,color=red] (0,0) to (0,-1.5);
\node[above] at (0,2.2){$G$};
\node[left] at (0,0){$\phi(x'')$};
\node[right] at (0,-1.5){$y=\phi(x')$};
\node at (-.4,1.2){$T(x)$};
\node[right] at (0,-.75){$\phi(B)$};
\node[right] at (.5,1){$\phi(x)$};
\begin{scope}[every node/.style={draw,shape=circle,fill=black,scale=.4}]
\node at (0,0){};
\node at (.5,1){};
\node at (0,1.5){};
\node at (0,-1.5){};
\node at (-1,.5){};
\end{scope}
\end{scope}
\end{tikzpicture}
\caption{Illustration of constructions used to prove connectedness of fibers of $q_F$.}
\label{fig:contour_proof}
\end{figure}
Let $y\in G_x\setminus T(x)$. Then there is an $x'\in \phi^{-1}(y)$ such that $x\in B(\psi(y),x')$.
Equivalently, $x'$ and $\psi(y)$ are in different connected components $F_1$ and $F_2$, respectively, of $F\setminus \{x\}$. (Since $y\notin C(x)$, neither $x'$ nor $\psi(y)$ is equal to $x$.)
Since $C(x)$ is closed, so is $\phi^{-1}(C(x))$. This means that we can pick an $x''\in \phi^{-1}(C(x))$ such that
\[
B\coloneqq B(x',x'')\setminus\{x''\}
\]
does not intersect $\phi^{-1}(C(x))$. It follows that $B\subseteq F_1$, since $x\notin B$.
It also follows that $\psi\circ\phi(B)\subseteq F_2$, since $x\notin \psi\circ\phi(B)$ and $\psi(y) \in \psi\circ\phi(B)$. Thus, for all $z\in B$, we have $x\in B(\psi\circ\phi(z),z)$; i.e.,
\[
(x,\phi(z))\in B(\psi\circ\phi(z),z)\times \{\phi(z)\}\subseteq Z,
\]
so $\phi(B)\subseteq G_x$. This means that there is a path in $G_x$ from $y=\phi(x')$ to $\phi(x'')\in T(x)$. Since $y$ was an arbitrary point in $G_x$ and $T(x)$ is connected, it follows that $G_x$ is connected.
\end{proof}

\section{Universality of the interleaving distance for merge trees}
\label{sec:merge}

In this section, we focus on merge trees, which are a special case of contour trees that also arise from the connected components of the sublevel set filtration of a function.

The merge trees obtained this way carry a function that is unbounded above, and they are characterized by the property that the canonical map from the merge tree to the Reeb graph of its epigraph is an isomorphism~\cite{MBW13}.
Our definition is more general and also admits bounded functions, and in \cref{sec:epi_merge} we develop an analogous characterization for these general merge trees via the property that said canonical map is an embedding.

Our goal in \cref{sec:interleaving_cont_merge} is to prove that the interleaving distance for merge trees is universal.
By \cref{thm:contour}, it suffices to construct a $\delta$-contortion pair from a $\delta$-interleaving of merge trees.
Summarizing the idea for the simpler special case of a merge tree $G$ unbounded above, the key insight behind the proof is that the $\delta$-smoothing of $G$ is isomorphic to a $\delta$-shift of $G$.
Composing the interleaving morphisms with the isomorphisms obtained this way yields the desired $\delta$-contortion pair in the unbounded case.

\subsection{Merge trees as Reeb graphs of epigraphs}
\label{sec:epi_merge}

We formally characterize merge trees using a construction
based on \emph{epigraphs}, as previously suggested by Morozov et al.~\cite{MBW13}.

\begin{definition}[Epigraph]
  Let $f \colon X \rightarrow \R$ be a continuous function.
  We define the \emph{epigraph of $f$} as the space
  $\mathcal{E} X := X \times [0, \infty)$
  equipped with the function
  ${\mathcal{E} f \colon \mathcal{E} X \rightarrow \R,\,
  (p, t) \mapsto f(p) + t.}$
\end{definition}

While this is not the usual definition of the epigraph
$\{(p, y) \in X \times \R \mid f(p) \leq y\}$,
we note that the map $\mathcal{E} f \colon \mathcal{E} X \rightarrow \R$
and the projection of the ordinary epigraph
to the second component
are isomorphic as $\R$-spaces.
Our definition has the benefit that we have the strict equality
\(
  \delta + \mathcal{E} f = \mathcal{E} (\delta + f)
  \colon \mathcal{E} X \rightarrow \R
\)
for any $\delta \in \R$.

  There is a natural embedding of $X$ into the epigraph of $f$, given by
  $\kappa_X \colon X \rightarrow X \times [0, \infty) = \mathcal{E} X,\,
  p \mapsto (p, 0)$.  
  It induces a map
    \begin{equation*}
    i_G \colon G
    \xto{\kappa_G} \mathcal{E} G
    \xto{q_{\mathcal{R} \mathcal{E} G}} \mathcal{R} \mathcal{E} G ,
  \end{equation*}
  which preserves function values in the sense that $\mathcal{R} \mathcal{E} f  \circ i_G  = f$. 
\begin{figure}[t]
  \centering
  \caption{
    The embedding $i_G \colon G \rightarrow \mathcal{R} \mathcal{E} G$
    of a merge tree $G$
    into its unbounded variant
    $\mathcal{R} \mathcal{E} G$.
  }
  \label{fig:mergeTreeEmbedding}
\end{figure}
  This map can be used to give the following alternative characterization of merge trees.

\begin{restatable}{proposition}{propCharMergeTrees}
  \label{prop:charMergeTrees}
  A Reeb graph $(G,g)$ is a merge tree
  iff 
  the map 
  $i_G \colon G \to \mathcal{R} \mathcal{E} G$
  is an embedding.  
\end{restatable}

\begin{proof}%
  Suppose $g \colon G \rightarrow \R$ is a Reeb graph
  that admits the structure of a CW complex with two distinct $1$-cells
  $e$ and $f$
  sharing their lower boundary point,
  and let $x$ and $y$ be interior points of $e$ and $f$, respectively,
  with $g(x) = g(y)$.
  Then
  $i_G(x) =
  i_G(y)$,
  and hence $q_{\mathcal{R} \mathcal{E} G} \circ \kappa_G$ is not injective.
  Thus, if $q_{\mathcal{R} \mathcal{E} G} \circ \kappa_G$ is an embedding,
  then $g \colon G \rightarrow \R$ is a merge tree.
  
  Now suppose that $g \colon G \rightarrow \R$ is a merge tree.
To show that $i_G$ is an embedding, it suffices to show that it is injective and open onto its image.

  \begin{claim}
    \label{claim:discreteFibers}
    The fibers of the induced map
    $\mathcal{R} \mathcal{E} g \colon \mathcal{R} \mathcal{E} G \rightarrow \R$
    are discrete.
  \end{claim}

  \begin{claimproof}
    As the sublevel sets of $g \colon G \rightarrow \R$ are locally connected,
    the function $\mathcal{E} g \colon \mathcal{E} G \rightarrow \R$
    has locally connected fibers.
    By \cite[Proposition 2.2]{bauer_edit} this implies
    that the fibers of
    $\mathcal{R} \mathcal{E} g \colon \mathcal{R} \mathcal{E} G \rightarrow \R$
    are discrete.    
  \end{claimproof}
  
  \begin{claim}
    \label{claim:Hausdorff}
    The quotient space $\mathcal{R} \mathcal{E} G$ is Hausdorff.
  \end{claim}

  \begin{claimproof}
    We have to show that any two distinct points
    in the same fiber $(\mathcal{R} \mathcal{E} g)^{-1}(t)$
    for some $t \in \R$
    admit disjoint neighbourhoods in $\mathcal{R} \mathcal{E} G$.
    To this end,
    we fix the structure of a CW complex on $G$
    as in \cref{def:ReebGraph}.
    Then the interlevel set
    $g^{-1} [t, t+1]$
    intersects a finite number of cells of this CW structure.
    Thus, there is a real number $u > t$ such that any cell intersecting
    $g^{-1} [t, u)$
    also intersects the fiber $g^{-1} (t)$.
    As a result,
    there is a retraction
    $r \colon g^{-1} [t, u) \rightarrow g^{-1} (t)$
    of $g^{-1} [t, u)$ onto $g^{-1} (t)$.
    From this retraction $r$
    we construct the retraction
    \begin{equation*}
      R \colon
      (\mathcal{E} g)^{-1} (-\infty, u) \rightarrow
      (\mathcal{E} g)^{-1} (t), \,
      (p, s) \mapsto
      \begin{cases}
        (p, t-g(p)) & p \in g^{-1} (-\infty, t] \\
        (r(p), 0) & p \in g^{-1} [t, u)
        .
      \end{cases}
    \end{equation*}
    As
    $(g^{-1} (-\infty, t] \times [0, \infty)) \cap
    (\mathcal{E} g)^{-1} (-\infty, u)$
    and
    $(g^{-1} [t, u) \times [0, \infty)) \cap
    (\mathcal{E} g)^{-1} (-\infty, u)$
    provide a closed cover
    of $(\mathcal{E} g)^{-1} (-\infty, u)$,
    the map
    $R \colon
    (\mathcal{E} g)^{-1} (-\infty, u) \rightarrow
    (\mathcal{E} g)^{-1} (t)$
    is indeed continuous.
    We consider the commutative diagram
    \begin{equation}
      \label{eq:HausdorffDiagram}
      \begin{tikzcd}[row sep=7ex]
        (\mathcal{E} g)^{-1} (-\infty, u)
        \arrow[rrr, "R"]
        \arrow[rd,
        "q_{\mathcal{R} \mathcal{E} G}"]%
        \arrow[dd, "\mathcal{E} g"']
        & & &
        (\mathcal{E} g)^{-1} (t)
        \arrow[rd, "q_{\mathcal{R} \mathcal{E} G}"]
        \arrow[dd, "\mathcal{E} g"' near start]
        \\
        &
        (\mathcal{R} \mathcal{E} g)^{-1} (-\infty, u)
        \arrow[rrr, crossing over, "\overline{R}" near start, dashed]
        \arrow[ld, "\mathcal{R} \mathcal{E} g"]
        & & &
        (\mathcal{R} \mathcal{E} g)^{-1} (t)
        \arrow[ld, "\mathcal{R} \mathcal{E} g"]
        \\
        (-\infty, u)
        \arrow[rrr]
        & & &
        \{t\}
        .
      \end{tikzcd}
    \end{equation}
    As the open subset
    $(\mathcal{E} g)^{-1} (-\infty, u) \subset \mathcal{E} G$
    is closed under the equivalence relation
    $\sim_{\mathcal{E} g}$
    from \cref{def:indReebGraph}
    the restricted map
    \begin{equation*}
      (\mathcal{E} g)^{-1} (-\infty, u)
      \xrightarrow{~q_{\mathcal{R} \mathcal{E} G}~}
      (\mathcal{R} \mathcal{E} g)^{-1} (-\infty, u)
    \end{equation*}
    is a quotient map as well.
    As a result, the induced map on quotients
    \[\overline{R} \colon
    (\mathcal{R} \mathcal{E} g)^{-1} (-\infty, u) \rightarrow
    (\mathcal{R} \mathcal{E} g)^{-1} (t)\]
    is continuous.
    Moreover, $\overline{R}$ is a retraction
    by the commutativity of \eqref{eq:HausdorffDiagram}.
    Now suppose we have $x \neq y \in (\mathcal{R} \mathcal{E} g)^{-1} (t)$.
    As $(\mathcal{R} \mathcal{E} g)^{-1} (t)$ is discrete
    by \cref{claim:discreteFibers},
    the singleton sets
    $\{x\}$ and $\{y\}$ are open in $(\mathcal{R} \mathcal{E} g)^{-1} (t)$.
    As a result, the preimages
    $\overline{R}^{-1} (x)$ and
    $\overline{R}^{-1} (y)$
    are disjoint open neighbourhoods of $x$ and $y$
    in $\mathcal{R} \mathcal{E} G$.
  \end{claimproof}

  \begin{claim}
    \label{claim:injective}
    The map $i_G \colon G \to \mathcal{R} \mathcal{E} G$
    is injective.
  \end{claim}

  \begin{claimproof}
First we fix the structure of a CW complex on $G$ as in \cref{def:mergeTree}.
Now let $t \in \R$ and let $x, y \in g^{-1}(t)$ with $i_G(x) = i_G(y)$.
We have to show that $x = y$.
By possibly subdividing two $1$-cells, we can assume that $x$ and $y$ are $0$-cells.
Since $i_G(x) = i_G(y)$, the points $x$ and $y$ are connected in $g^{-1} (-\infty, t]$.
As $g^{-1} (-\infty, t]$ is locally path-connected by \cref{rem:pathConn}, this implies that there is a path $\gamma \colon [0, 1] \rightarrow g^{-1} (-\infty, t]$ from $x$ to $y$.
Moreover, $g^{-1} ((g \circ \gamma)[0,1])$ intersects a finite number of cells by \cref{def:ReebGraph}.
Thus, we can choose a finite subcomplex $G'$ of $G$ that contains $\im \gamma$.
Since $x$ and $y$ are connected in $G'$, they are connected when we view $G'$ as an undirected graph with the $0$-cells as vertices and the $1$-cells as edges.
Let $\gamma'$ be a minimal path from $x$ to $y$ in this undirected graph.
Since edges in a Reeb graph always connect vertices with different values under $g$ and there is no vertex $z$ with $g(z)>g(x)=g(y)=t$, there must be a pair of incident edges $(u,v)$ and $(v,w)$ with $g(u)>g(v)<g(w)$.
But this contradicts the definition of merge tree.
  \end{claimproof}

  \begin{claim}
  \label{claim:open}
  $i_G$ is an open map onto its image.
  \end{claim}
  
  \begin{claimproof}
  To show that $i_G$ is open onto its image, it suffices to show that for any open $U$ with bounded range $g(U)$, the image $i_G(U)$ is open in $\im i_G$, as those open sets form a basis of the topology.

  Now $U$ is an open subset of $g^{-1}(a,b)$ of $G$ for some $a<b$.
  Since $g^{-1}[a,b]$ is closed and intersects only finitely many cells of $G$, it is compact.
  Thus, as $\mathcal{R} \mathcal{E} G$ is Hausdorff by \cref{claim:Hausdorff} and $i_G$ is injective by \cref{claim:injective}, the restriction of $i_G$ to $g^{-1}[a,b]$ is an embedding.
  This means that $i_G(U)$ is open in $i_G(g^{-1}[a,b]) = i_G(G) \cap (\mathcal{R} \mathcal{E} g)^{-1}[a,b]$, where the equality comes from $i_G$ being function-preserving.
  It follows that $i_G(U)$ is open in $i_G(G)\cap (\mathcal{R} \mathcal{E} g)^{-1}(a,b)$ and thus also in $i_G(G)$, as $(\mathcal{R} \mathcal{E} g)^{-1}(a,b)$ is open in $\mathcal{R} \mathcal{E} G$.
\end{claimproof}

  By \cref{claim:open} the map $i_G$
  is open onto its image, and by \cref{claim:injective} it is injective,
  hence it is an embedding.
\end{proof}

\subsection{Bounded and unbounded merge trees}

The map $i_G$ already provides an embedding of any bounded merge tree $G$ into the unbounded merge tree $\mathcal{R} \mathcal{E} G$. We complement this construction with a retraction from $\mathcal{R} \mathcal{E} G$ to $G$.

First suppose that $(G,g)$ is a Reeb graph,
and let $m := \sup_{p \in G} g(p) \in (-\infty, \infty]$.
We define the map
\[\tilde{\rho}_G \colon \mathcal{E} G \rightarrow \mathcal{E} G,\,
  (p, t) \mapsto (p, \min \{t, m-g(p)\})
,\]
which makes the diagram
\begin{equation*}
  \begin{tikzcd}[column sep=8ex]
    (\mathcal{E} g)^{-1} (-\infty, m]
    \arrow[r, equals]
    \arrow[d, hook]
    &
    (\mathcal{E} g)^{-1} (-\infty, m]
    \arrow[d, hook]
    \\
    \mathcal{E} G
    \arrow[r,dashed, "\tilde{\rho}_G"]
    \arrow[d, "\mathcal{E} g"']
    &
    \mathcal{E} G
    \arrow[d, "\mathcal{E} g"]
    \\
    \R
    \arrow[r, "{\min \{-, m\}}"']
    &
    \R
  \end{tikzcd}
\end{equation*}
commute.
We state the following immediate consequence of this definition. 
\begin{lemma}
  \label{lem:rhoTwiddleFibers}
  For each $t \in \R$
  the map
  $\tilde{\rho}_G \colon \mathcal{E} G \rightarrow \mathcal{E} G$
  restricts to a homeomorphism
  between the fibers $(\mathcal{E} g)^{-1}(t)$
  and $(\mathcal{E} g)^{-1}(\min \{t, m\})$.
\end{lemma}
By the universal property of the quotient topology,
there is a unique continuous map
${\mathcal{R} \tilde{\rho}_G \colon
\mathcal{R} \mathcal{E} G \rightarrow \mathcal{R} \mathcal{E} G}$
making the following diagram commute:
\begin{equation}
  \label{eq:rhoTwiddleDgm}
  \begin{tikzcd}[row sep=7ex]
    \mathcal{E} G
    \arrow[rrr, "\tilde{\rho}_G"]
    \arrow[rd, "q_{\mathcal{R} \mathcal{E} G}"]
    \arrow[dd, "\mathcal{E} g"']
    & & &
    \mathcal{E} G
    \arrow[rd, "q_{\mathcal{R} \mathcal{E} G}"]
    \arrow[dd, "\mathcal{E} g"' near start]
    \\
    &
    \mathcal{R} \mathcal{E} G
    \arrow[rrr, crossing over, "\mathcal{R} \tilde{\rho}_G" near start]
    \arrow[ld, "\mathcal{R} \mathcal{E} g"]
    & & &
    \mathcal{R} \mathcal{E} G
    \arrow[ld, "\mathcal{R} \mathcal{E} g"]
    \\
    \R
    \arrow[rrr, "{\min \{-, m\}}"']
    & & &
    \R
  \end{tikzcd}
\end{equation}

\begin{corollary}
  \label{lem:rhoReebFibers}
  For each $t \in \R$
  the map
  $\mathcal{R} \tilde{\rho}_G \colon
  \mathcal{R} \mathcal{E} G \rightarrow \mathcal{R} \mathcal{E} G$
  restricts to a bijection
  between the fibers $(\mathcal{R} \mathcal{E} g)^{-1}(t)$
  and $(\mathcal{R} \mathcal{E} g)^{-1}(\min \{t, m\})$.
\end{corollary}

\begin{restatable}{lemma}{lemImages}
  \label{lem:images}
  The images of the maps 
  $i_G \colon G \to \mathcal{R} \mathcal{E} G$ and $\mathcal{R}\tilde{\rho}_G \colon \mathcal{R} \mathcal{E} G \to \mathcal{R} \mathcal{E} G$
  are identical.
\end{restatable}

\begin{proof}
  By definition of $m = \sup_{p \in G} g(p)$,
  the image of the map $i_G \colon G \to \mathcal{R} \mathcal{E} G$
  is contained in the image of $\mathcal{R} \tilde{\rho}_G$.
  Now we consider a point
  $q_{\mathcal{R} \mathcal{E} G} (p, t) \in \mathcal{R} \mathcal{E} G$
  with
  $m > (\mathcal{R} \mathcal{E} g)(q_{\mathcal{R} \mathcal{E} G} (p, t)) =
  (\mathcal{E} g)(p, t) = g(p) + t$.
  We have to show that there is a point $\tilde{p} \in G$ with
  $i_G(\tilde{p}) =
  q_{\mathcal{R} \mathcal{E} G} (p, t)$.
  To this end,
  let $p' \in g^{-1} [g(p) + t, m]$.
  Since $G$ is connected, there is a path $\gamma$
  from $p$ to $p'$.
  Now let $\tilde{p}$ be the first point on $\gamma$
  with $g(\tilde{p}) = g(p) + t$.
  Then the segment of $\gamma$ connecting $p$ and $\tilde{p}$
  is entirely contained in $g^{-1} (-\infty, g(p) + t]$,
  and thus
  $i_G(\tilde{p}) =
  q_{\mathcal{R} \mathcal{E} G} (p, t)$.
  In case $m = \infty$ this completes our proof.
  Now suppose $m \in \R$.
  As $G$ admits the structure of a CW complex
  with the properties from \cref{def:ReebGraph},
  the function $g \colon G \rightarrow \R$ is proper,
  and hence there is a point $p \in G$ with $g(p) = m$.
  Moreover, as $G$ is connected, we have
  $(\mathcal{R} \mathcal{E} g)^{-1} (m) =
  \{i_G(p)\}$.
\end{proof}

Now suppose that $(G,g)$ is a merge tree.
  The map $i_G$
  is non-surjective
  iff $g \colon G \rightarrow \R$ is bounded above.
  We define
  $\rho_G \colon \mathcal{R} \mathcal{E} G \rightarrow G$
  to be the unique continuous map
  -- which exists by \cref{lem:images,prop:charMergeTrees} --
  making the diagram
  \begin{equation}
    \label{eq:rhoDgm}
    \begin{tikzcd}[row sep=5ex, column sep=8ex]
      G
      \arrow[d, "\kappa_G"']
      &
      \mathcal{R} \mathcal{E} G
      \arrow[l, dashed, "\rho_G"']
      \arrow[d, "\mathcal{R} \tilde{\rho}_G"]
      \\
      \mathcal{E} G
      \arrow[r, "q_{\mathcal{R} \mathcal{E} G}"']
      &
      \mathcal{R} \mathcal{E} G
    \end{tikzcd}
  \end{equation}
 commute.
 As an immediate corollary of \cref{lem:rhoReebFibers}, we obtain the following observation.
\begin{corollary}
  \label{lem:rhoFibers}
  For each $t \in \R$,
  the map
  $\rho_G \colon \mathcal{R} \mathcal{E} G \rightarrow G$
  restricts to a bijection
  between the fibers $(\mathcal{R} \mathcal{E} g)^{-1}(t)$
  and $g^{-1}(\min \{t, m\})$.
\end{corollary}

\subsection{Interleavings, contortions, and merge trees}
\label{sec:interleaving_cont_merge}

Let $f \colon X \rightarrow \R$ be an arbitrary continuous function
and let $\delta \geq 0$.
We define the embedding
\[\kappa^{\delta}_X \colon \mathcal{T}_{\delta} X \rightarrow \mathcal{E} X,\,
  (p, t) \mapsto (p, t+\delta),\]
which makes the diagram
\begin{equation*}
  \begin{tikzcd}[column sep=7ex]
    \mathcal{T}_{\delta} X
    \arrow[r, "\kappa^{\delta}_X"]
    \arrow[d, "\mathcal{T}_{\delta} f"']
    &
    \mathcal{E} X
    \arrow[d, "\mathcal{E} f"]
    \\
    \R
    \arrow[r, "(-)+\delta"']
    &
    \R
  \end{tikzcd}
\end{equation*}
commute.
Now let $(G, g)$ be a merge tree.
By the universal property of the quotient topology,
there is a unique continuous map
${\mathcal{R} \kappa^{\delta}_G \colon
\mathcal{U}_{\delta} G \rightarrow \mathcal{R} \mathcal{E} G}$
making the diagram
\begin{equation}
  \label{eq:kappaDeltaDgm}
  \begin{tikzcd}[row sep=7ex]
    \mathcal{T}_{\delta} G
    \arrow[rrr, "\kappa^{\delta}_G"]
    \arrow[rd, "q_{\mathcal{U}_{\delta} G}"]
    \arrow[dd, "\mathcal{T}_{\delta} g"']
    & & &
    \mathcal{E} G
    \arrow[rd, "q_{\mathcal{R} \mathcal{E} G}"]
    \arrow[dd, "\mathcal{E} g"' near start]
    \\
    &
    \mathcal{U}_{\delta} G
    \arrow[rrr, crossing over,
    "\mathcal{R} \kappa^{\delta}_G" near start]
    \arrow[ld, "\mathcal{U}_{\delta} g"]
    & & &
    \mathcal{R} \mathcal{E} G
    \arrow[ld, "\mathcal{R} \mathcal{E} g"]
    \\
    \R
    \arrow[rrr, "(-)+\delta"']
    & & &
    \R
  \end{tikzcd}
\end{equation}
commute.

\begin{restatable}{lemma}{lemKappaReebInj}
  \label{lem:kappaReebInj}
  The map
  ${\mathcal{R} \kappa^{\delta}_G \colon
    \mathcal{U}_{\delta} G \rightarrow \mathcal{R} \mathcal{E} G}$
  is injective.
\end{restatable}

\begin{proof}
In the following commutative diagram, $\kappa_G$ is the composition of the top horizontal maps, so we get $i_G$ by following any path in the diagram from $G$ to $\mathcal{RE} G$.
Taking the lower path, we get that $i_G$ factors through $\mathcal{R} \kappa^{\delta}_G$.
\begin{equation*}
\begin{tikzcd}[column sep=7ex]
G
\arrow[r, "(-{,}\delta)"]
&
\mathcal{T}_{\delta} G
\arrow[r, "\kappa^{\delta}_G"]
\arrow[d, "q_{\Uu_\delta G}"']
&
\mathcal{E} G
\arrow[d, "q_{\mathcal{RE} G}"]
\\
&
\mathcal U_\delta G
\arrow[r, "\mathcal R \kappa^{\delta}_G"']
&
\mathcal{RE} G
\end{tikzcd}
\end{equation*}
Since $i_G$ is injective by \cref{claim:injective}, it is enough to check that $\Rr\kappa^\delta_G(x)= \Rr\kappa^\delta_G(y)$ implies $x=y$ for $x\notin \im q_{\Uu_\delta G}\circ (-,\delta)$.
Let $t=\Uu_\delta g(x)$, $X = q_{\Uu_\delta G}^{-1}(x)$ and $Y = q_{\Uu_\delta G}^{-1}(y)$, and let $\pi_2\colon G\times \R \to \R$ be defined by $(p,t)\mapsto t$.
In this case, $X \subseteq \T_\delta g^{-1}(t)$ does not contain any point of the form $(\tilde x,\delta)$, so $X\subseteq \pi_2^{-1}[-\delta, \delta)$.
Moreover, by definition of $q_{\Uu_\delta G}$, $X$ is a connected component in $\T_\delta g^{-1}(t)$ and hence open and closed.

Note that $\kappa^\delta_G$ maps $\T_\delta g^{-1}(t)$ homeomorphically onto $\pi_2^{-1}[0, 2 \delta]\cap \Ee g^{-1}(t+\delta)$, and $\pi_2^{-1}[-\delta,\delta)\cap \T_\delta g^{-1}(t)$ homeomorphically onto
\[
\pi_2^{-1}[0, 2 \delta)\cap \Ee g^{-1}(t+\delta)=\pi_2^{-1}(-\infty, 2 \delta)\cap \Ee g^{-1}(t+\delta).
\]
Since $X$ is closed in $\T_\delta g^{-1}(t)$ and $\pi_2^{-1}[0, 2 \delta]\cap \Ee g^{-1}(t+\delta)$ is closed in $\Ee g^{-1}(t+\delta)$, the former means that $\kappa^\delta_G(X)$ is closed in $\Ee g^{-1}(t+\delta)$.
Since $X \subseteq \pi_2^{-1}[-\delta, \delta)$ is open in $\T_\delta g^{-1}(t)$ and $\pi_2^{-1}(-\infty, 2 \delta)\cap \Ee g^{-1}(t+\delta)$ is open in $\Ee g^{-1}(t+\delta)$, the latter means that $\kappa^\delta_G(X)$ is open in $\Ee g^{-1}(t+\delta)$.
Since $X$ is connected, its image $\kappa^\delta_G(X)$ is connected, open, and closed, and hence a connected component of $\Ee g^{-1}(t+\delta)$.
Since $\Rr\kappa^\delta_G(x)= \Rr\kappa^\delta_G(y)$, this implies $\kappa^\delta_G(Y)\subseteq \kappa^\delta_G(X)$ and thus $Y\subseteq X$ by injectivity of $\kappa^\delta_G$.
It follows that $x=y$, which was what we wanted to prove.
\end{proof}

\begin{restatable}{lemma}{lemThickeningShift}
  \label{lem:thickeningShift}
  As in the previous subsection,
  let $m := \sup_{p \in G} g(p) \in (-\infty, \infty]$,
  let $p \in G$,
  and let $\tau = -\delta$
  or $\tau \in [-\delta, \delta]$ if $g(p) = m$.
  The composite map
  \begin{equation*}
    G \times [-\delta, \delta] = \mathcal{T}_{\delta} G
    \xto{q_{\mathcal{U}_{\delta} G}}
    \mathcal{U}_{\delta} G
    \xto{\mathcal{R} \kappa^{\delta}_G}
    \mathcal{R}\mathcal{E} G
    \xto{\rho_G}
    G
  \end{equation*}
  maps $(p, \tau)$ to $p$.
\end{restatable}

\begin{proof}%
  Let $p' \in G$ be the image of $(p, \tau)$
  under the composite map
  \begin{equation*}
    G \times [-\delta, \delta] = \mathcal{T}_{\delta} G
    \xto{q_{\mathcal{U}_{\delta} G}}
    \mathcal{U}_{\delta} G
    \xto{\mathcal{R} \kappa^{\delta}_G}
    \mathcal{R} \mathcal{E} G
    \xto{\rho_G}
    G
    .
  \end{equation*}
  We have to show that $p = p'$.
  As $g \colon G \rightarrow \R$ is a merge tree, the  map
  $i_G \colon G \to \mathcal{R} \mathcal{E} G$
  is an embedding by \cref{prop:charMergeTrees}.
  Thus,
  it suffices to show that $p$ and $p'$ are identified by
  $i_G$.
  To this end,
  we consider the commutative diagram
  \begin{equation*}
    \begin{tikzcd}[column sep=10ex, row sep=7ex]
      \mathcal{T}_{\delta} G
      \arrow[r, "\kappa^{\delta}_G"]
      \arrow[d, "q_{\mathcal{U}_{\delta} G}"']
      &
      \mathcal{E} G
      \arrow[r, "\tilde{\rho}_G"]
      \arrow[d, "q_{\mathcal{R} \mathcal{E} G}"]
      &
      \mathcal{E} G
      \arrow[d, "q_{\mathcal{R} \mathcal{E} G}"]
      \\
      \mathcal{U}_{\delta} G
      \arrow[r, "\mathcal{R} \kappa^{\delta}_G"']
      &
      \mathcal{R} \mathcal{E} G
      \arrow[r, "\mathcal{R} \tilde{\rho}_G"']
      \arrow[d, "\rho_G"']
      &
      \mathcal{R} \mathcal{E} G
      \\
      &
      G
      \arrow[r, "\kappa_G"']
      \arrow[ru, "i_G"']
      &
      \mathcal{E} G .
      \arrow[u, "q_{\mathcal{R} \mathcal{E} G}"']
    \end{tikzcd}
  \end{equation*}
  Now the descending stair on the left of the diagram
  maps $(p, \tau) \in \mathcal{T}_{\delta} G$ to $p' \in G$.
  So we have to show that chasing $(p, \tau)$
  from the upper left corner $\mathcal{T}_{\delta} G$
  to $\mathcal{R} \mathcal{E} G$
  at the end of the second line
  and chasing $p \in G$ from center south
  to $\mathcal{R} \mathcal{E} G$
  at the end of the second line
  we obtain the same element.
  Now $\kappa^{\delta}_G$ maps $(p, \tau)$
  to $(p, \tau + \delta)$
  and $\tilde{\rho}_G$ maps $(p, \tau + \delta)$
  to $(p, 0)$ by our assumptions on $\tau$.
  Furthermore, $(p, 0)$ is mapped to
  $q_{\mathcal{R} \mathcal{E} G} (p, 0) \in \mathcal{R} \mathcal{E} G$.
  Moreover, $i_G$ also maps $p$ to
  $q_{\mathcal{R} \mathcal{E} G} (p, 0) \in \mathcal{R} \mathcal{E} G$.
  Thus, $(p, \tau)$ and $p$ are mapped to the same element
  in $\mathcal{R} \mathcal{E} G$,
  hence $p' = p$.
\end{proof}

\begin{proof}[Proof of \cref{thm:merge}]
  Suppose $(F,f)$ and $(G,g)$
  are merge trees 
  and that
  \[\phi \colon F \rightarrow \mathcal{U}_{\delta} G
    \quad \text{and} \quad
    \psi \colon G \rightarrow \mathcal{U}_{\delta} F
  \]
  form a $\delta$-interleaving (of Reeb graphs).
  We show that the composite maps
  \begin{align*}
    \tilde{\phi} \colon
    & F \xto{\phi}
      \mathcal{U}_{\delta} G \xto{
      \mathcal{R} \kappa^{\delta}_G
      }
      \mathcal{R} \mathcal{E} G \xto{\rho_G}
      G
    , &
    \tilde{\psi} \colon
    & G \xto{\psi}
      \mathcal{U}_{\delta} F \xto{
      \mathcal{R} \kappa^{\delta}_F
      }
      \mathcal{R} \mathcal{E} F \xto{\rho_F}
      F
  \end{align*}
  form a $\delta$-contortion pair. Together with \cref{thm:contour}, this proves the claim.

  Let $x \in F$ and let $y \in \tilde{\psi}^{-1} (x)$.
  We have to show that
  $y$ and $\tilde{\phi}(x)$ are connected in $g^{-1} [f(x)\pm\delta]$.
  By the symmetry of \cref{def:contortionDistance}, this is also sufficient.
  By the commutativity of the lower parallelogram in \eqref{eq:kappaDeltaDgm}
  the value of
  $\mathcal{R} \kappa^{\delta}_F (\psi(y))$
  under $\mathcal{R} \mathcal{E} f$ is
  \[(\mathcal{U}_{\delta} f)(\psi(y)) + \delta = g(y) + \delta.\]
  In conjunction with the commutativity of \eqref{eq:rhoDgm}
  and the lower parallelogram in \eqref{eq:rhoTwiddleDgm}
  we obtain
  \[f(x) = \big(f \circ \tilde{\psi}\big)(y) = \min \{g(y) + \delta, m'\},\]
  where $m' := \sup_{p \in F} f(p)$,
  and hence
  \[f(x) - g(y) =
    \min \{g(y) + \delta, m'\} - g(y) =
    \min \{\delta, m' - g(y)\}.
  \]
  Moreover,
  $g(y) = (\mathcal{U}_{\delta} f) (\psi(y)) \leq m' + \delta$,
  so in conjunction with \cref{lem:thickeningShift} we obtain that
  \begin{equation*}
    F \times [-\delta, \delta] = \mathcal{T}_{\delta} F
    \xto{q_{\mathcal{U}_{\delta} F}}
    \mathcal{U}_{\delta} F
    \xto{\mathcal{R} \kappa^{\delta}_F}
    \mathcal{R}\mathcal{E} F
    \xto{\rho_F}
    F
  \end{equation*}
  maps $(x, g(y) - f(x))$ to $x$.
  Thus, the composite map
  \begin{equation*}
    \mathcal{U}_{\delta} F
    \xto{\mathcal{R} \kappa^{\delta}_F}
    \mathcal{R}\mathcal{E} F
    \xto{\rho_F}
    F
  \end{equation*}
  maps both
  $q_{\mathcal{U}_{\delta} F} (x, g(y) - f(x))$ and $\psi(y)$
  to $x$.
  By \cref{lem:kappaReebInj,lem:rhoFibers} this implies
  \begin{equation*}
    q_{\mathcal{U}_{\delta} F} (x, g(y) - f(x)) = \psi(y)
    .
  \end{equation*}
  Completely analogously we obtain that
  \(
    q_{\mathcal{U}_{\delta} G}
    \big(\tilde{\phi}(x), f(x) - \big(g \circ \tilde{\phi}\big)(x)\big) =
    \phi(x)
    .
  \)
  Thus, $y$~and~$\tilde{\phi}(x)$ are connected in
  $g^{-1} [g(y)\pm2\delta]$
  by \cref{def:interleaving}.
  It remains to show that
  $y$ and $\tilde{\phi}(x)$ are connected in $g^{-1} [f(x)\pm\delta]$.
  To this end,
  let $t := \min \{g(y) + 2 \delta, m\}$,
  where $m := \sup_{p \in G} g(p)$.

  \begin{claim}
    \label{claim:imageFiber}
    We have
    $\big(g \circ \tilde{\phi}\big)(x) = t$.
  \end{claim}

  \begin{claimproof}
    \sloppy
    We first consider the case
    ${\big(f \circ \tilde{\psi}\big)(y) = f(x) < m'}$.
    In this case, we have
    ${f(x) = \big(f \circ \tilde{\psi}\big)(y) = g(y) + \delta}$ and thus
    ${\big(g \circ \tilde{\phi}\big)(x) = t}$.
    Now suppose ${f(x) = m'}$.
    Since ${\tilde{\phi}(x) \in g^{-1} [g(y)\pm\delta]}$,
    we must have
    ${\big(g \circ \tilde{\phi}\big)(x) \leq t}$.
    Now suppose
    ${\big(g \circ \tilde{\phi}\big)(x) < t \leq m}$. Then ${\big(g \circ \tilde{\phi}\big)(x) = f(x) + \delta = m' + \delta}$.
    In particular, we have ${m' + \delta < t \leq m}$.
    Now let ${s \in (m' + \delta, m)}$.
    Then we have
    ${(\mathcal{U}_{\delta} f)^{-1} (s) = \emptyset}$ while
    ${g^{-1}(s) \neq \emptyset}$,
    a contradiction to the existence of the map
    ${\psi |_{g^{-1}(s)} \colon
      g^{-1}(s) \rightarrow (\mathcal{U}_{\delta} f)^{-1} (s)}$.
  \end{claimproof}

  We obtain the connectivity of
  $y$ and ${\tilde{\phi}(x)}$
  in ${g^{-1} [f(x)\pm\delta]}$
  from their connectivity in
  ${g^{-1} [g(y)\pm2\delta]}$
  by defining a retraction
  ${\sigma \colon g^{-1} [g(y)\pm2\delta] \rightarrow g^{-1}(t)}$
  using the map
  \begin{equation*}
    \tilde{\sigma} \colon
    (p, s) \mapsto (p, \max \{s, t - g(p)\})
    .
  \end{equation*}
  as a composition of maps
  \begin{equation*}
    \begin{tikzcd}
      g^{-1} [g(y)\pm2\delta]
      \arrow[d, hook]
      \\
      g^{-1} (-\infty, g(y) + 2\delta]
      \arrow[d, "\kappa_G"]
      \\
      (\mathcal{E} g)^{-1} (-\infty, g(y) + 2\delta]
      \arrow[d, "\tilde{\sigma}"]
      \\
      (\mathcal{E} g)^{-1} [t, g(y) + 2\delta]
      \arrow[d, "q_{\mathcal{R} \mathcal{E} G}"]
      \\
      (\mathcal{R} \mathcal{E} g)^{-1} [t, g(y) + 2\delta]
      \arrow[d, "\rho_G"]
      \\
      g^{-1}(t)
      .     
    \end{tikzcd}
  \end{equation*}
  By the definition of $\tilde{\sigma}$ the map
  ${\sigma \colon g^{-1} [g(y)\pm2\delta] \rightarrow g^{-1}(t)}$
  is indeed a retraction.
  As ${\tilde{\phi}(x) \in g^{-1}(t)}$ by \cref{claim:imageFiber}
  the points ${\sigma(y)}$ and ${\tilde{\phi}(x)}$
  are connected in the fiber ${g^{-1}(t)}$.
  Since the fibers of $g$ are discrete, this implies that
  ${\sigma(y) = \tilde{\phi}(x)}$.
  Defining the path
  \begin{equation*}
    \gamma \colon [0, 1] \rightarrow \mathcal{E} G, \,
    s \mapsto (y, s(t-g(y)))
  \end{equation*}
  the composition
  \begin{equation*}
    [0, 1] \xto{\gamma}
    \mathcal{E} G \xto{q_{\mathcal{R} \mathcal{E} G}}
    \mathcal{R} \mathcal{E} G \xto{\rho_G}
    G
  \end{equation*}
  yields a path from $y$ to $\sigma(y) = \tilde{\phi}(x)$
  in $g^{-1} [f(x)\pm\delta]$.
\end{proof}

\bibliography{Reebgd}

\end{document}